\documentclass[12pt]{amsart}
\usepackage{graphics,color,epsfig}
\usepackage{amsfonts,amssymb}

\newcommand{\Q}{\mathbb Q}
\newcommand{\R}{\mathbb R}
\newcommand{\Z}{\mathbb Z}

\DeclareMathOperator{\area}{area}
\DeclareMathOperator{\diam}{diam}
\DeclareMathOperator{\dist}{dist}
\DeclareMathOperator{\DIV}{DIV}
\DeclareMathOperator{\Hdim}{HDim}
\DeclareMathOperator{\hol}{hol}
\DeclareMathOperator{\hor}{hor}
\DeclareMathOperator{\NE}{NE}
\DeclareMathOperator{\SL}{SL}

\newcommand{\ds}{\displaystyle}
\newcommand{\eps}{\varepsilon}
\newcommand{\vhi}{\varphi}
\newcommand{\pmat}[1]{\begin{pmatrix}#1\end{pmatrix}}

\newcommand{\impl}{\quad\Longrightarrow\quad}

\newtheorem{theorem}{Theorem}[section]
\newtheorem{corollary}[theorem]{Corollary}
\newtheorem{lemma}[theorem]{Lemma}
\newtheorem{proposition}[theorem]{Proposition}

\theoremstyle{definition}
\newtheorem{definition}[theorem]{Definition}
\newtheorem{notation}[theorem]{Notation}

\theoremstyle{remark}
\newtheorem{remark}[theorem]{Remark}

\begin{document}
\title[Dichotomy of Hausdorff dimension]
{Dichotomy for the Hausdorff dimension 
of the set of nonergodic directions}
\date{\today}
\author{Yitwah Cheung, Pascal Hubert, Howard Masur}
\thanks{First and third authors supported by NSF DMS-0701281 and DMS-0905907, 
respectively and second author supported by ANR-06-BLAN-0038.}  

\begin{abstract}
Given an irrational $0<\lambda<1$, we consider billiards in the table $P_\lambda$ 
  formed by a $\tfrac12\times1$ rectangle with a horizontal barrier of length 
  $\frac{1-\lambda}{2}$ with one end touching at the midpoint of a vertical side.  
Let $\NE(P_\lambda)$ be the set of $\theta$ such that the flow on $P_\lambda$ in 
  direction $\theta$ is not ergodic.  
We show that the Hausdorff dimension of $\NE(P_\lambda)$ can only take on 
  the values $0$ and $\tfrac12$, depending on the summability of the series 
  $\sum_k \frac{\log \log q_{k+1}}{q_k}$ where $\{q_k\}$ is the sequence of 
  denominators of the continued fraction expansion of $\lambda$.  
More specifically, we prove that the Hausdorff dimension is $\tfrac12$ if this 
  series converges, and $0$ otherwise.  
This extends earlier results of Boshernitzan and Cheung.  
\end{abstract}
\maketitle

\section{Introduction}
In 1969, (\cite{Ve1}) Veech found examples of skew products over a rotation 
  of the circle that are minimal but not uniquely ergodic.  These were turned 
  into interval exchange transformations in \cite{KN}.  Masur and Smillie gave 
  a geometric interpretation of these examples (see for instance \cite{MT}) 
  which may be described as follows.  
Let $P_\lambda$ denote the billiard in a $\tfrac12\times1$ rectangle with a 
  horizontal barrier of length $\alpha=\tfrac{1-\lambda}{2}$ based at the 
  midpoint of a vertical side.  There is a standard unfolding procedure which 
  turns billiards in this polygon into flows along parallel lines on a translation 
  surface.  See Figure~\ref{fig:unfolding}.  

\begin{figure}[ht]\label{fig:unfolding}
\includegraphics{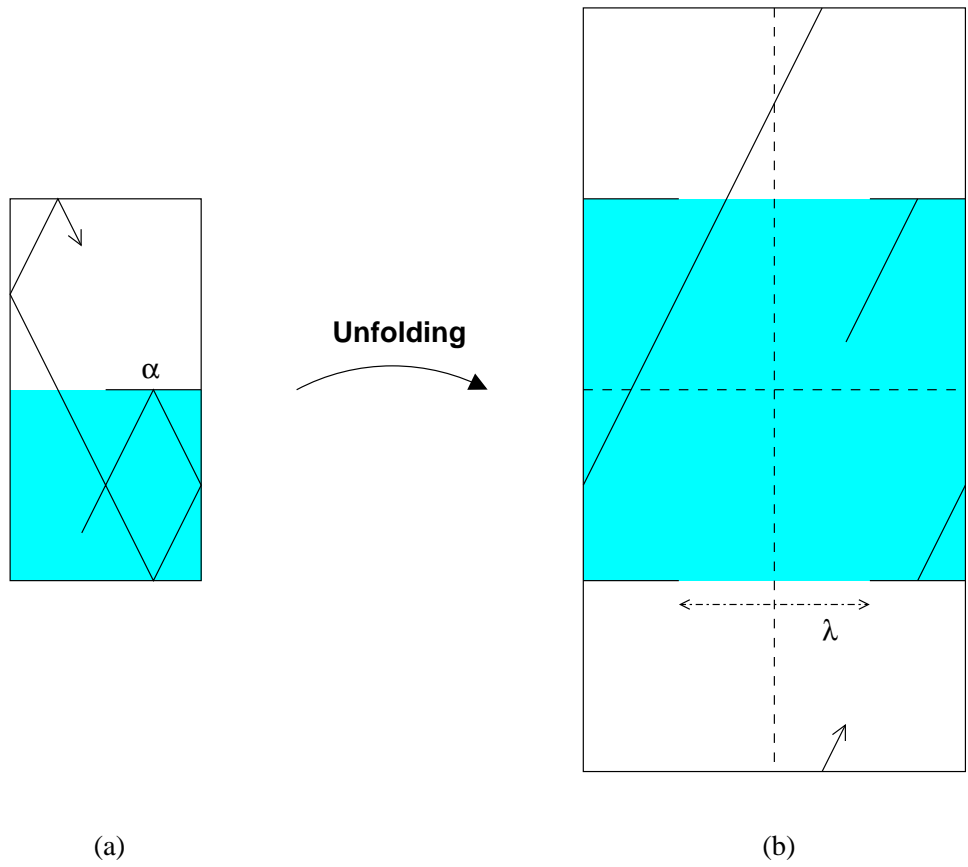}  
\caption{Unfolding the table $P_\lambda$.}
\end{figure}

The associated translation surface in this case is a double cover of a standard 
flat torus of area one branched over two points $z_0$ and $z_1$ a horizontal 
distance $\lambda$ apart on the flat torus.   See Figure~\ref{fig:slit-torus}. 
We denote it by $(X,\omega)$.  

\begin{figure}[ht]\label{fig:slit-torus}
\includegraphics{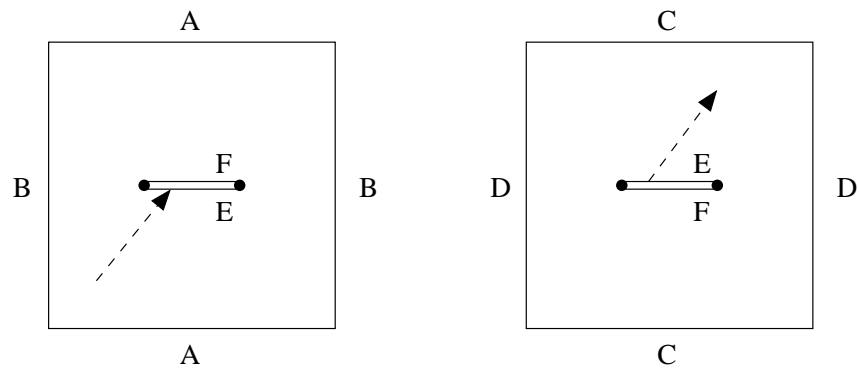}  
\caption{The branched double cover $(X,\omega)$.}
\end{figure}

The linear flows on this translation surface preserve Lebesgue measure.  
What Veech showed in these examples is that given $\theta$ with unbounded 
partial quotients in its continued fraction expansion, there is a $\lambda$ 
such that the flow on $P_\lambda$ in direction with slope $\theta$ is minimal 
but not uniquely ergodic.  

Let $\NE(P_\lambda)$ denote the set of nonergodic directions, i.e. those directions 
  for which Lebesgue measure is not ergodic.  It was shown in \cite{MT} that 
  $\NE(P_\lambda)$ is uncountable if $\lambda$ is irrational.  
When $\lambda$ is rational, a result of Veech (\cite{Ve2}) implies that minimal 
  directions are uniquely ergodic; thus $\NE(P_\lambda)$ is the set of rational 
  directions and is countable.  
By a general result of Masur (see \cite{Ma2}), the Hausdorff dimension of 
  $\NE(P_\lambda)$ satisfies $\Hdim \NE(P_\lambda) \le \tfrac12$.  

In \cite{Ch1} Cheung proved that this estimate is sharp.  He showed that if 
  $\lambda$ is \emph{Diophantine}, then $\Hdim \NE(P_\lambda) \ge \tfrac12$.  
Recall that $\lambda$ is Diophantine if there is lower bound of the form 
  $$\left|\lambda-\frac{p}{q}\right|>\frac{c}{q^s}, \quad c>0,s>0$$ 
  controlling how well $\lambda$ can be approximated by rationals.  
This raises the question of the situation when $\lambda$ is irrational but 
  not Diophantine; namely, when $\lambda$ is a Liouville number.  
Boshernitzan showed that $\Hdim\NE(P_\lambda)=0$ for a residual (in particular, 
  uncountable) set of $\lambda$ (see the Appendix in \cite{Ch1}) although 
  it is not obvious how to exhibit a specific Liouville number in this set.  

In this paper, we establish the following dichotomy:
\begin{theorem}\label{thm:dichotomy}
Let $\{q_k\}$ be the sequence of denominators in the continued fraction 
  expansion of $\lambda$.  Then $\Hdim\NE(P_\lambda)=0$ or $\tfrac12$, 
  the latter case occurring if and only if $\lambda$ is irrational and 
\begin{equation}\label{PM:conv}
  \sum_k \frac{\log\log q_{k+1}}{q_k} < \infty.\footnote{This condition on the 
  denominators of the continued fraction appears in complex dynamics in the 
  work of P\'erez Marco (\cite{PM}) in the context of the linearization problem 
  and is commonly referred to as the P\'erez Marco condition.  An expository 
  account of the history leading up to this work is given in \cite{Mi}.}  
\end{equation}
\end{theorem}

We briefly outline the proof of Theorem~\ref{thm:dichotomy}, which naturally 
  divides into two parts: an upper bound argument giving the dimension $0$ result and a  lower bound argument giving the dimension $\tfrac12$ 
  result.   
In \S\ref{s:NEdir} we discuss the geometry of the surface $(X,\omega)$ associated to $P_\lambda$; 
  in particular, the ways it can be decomposed into tori glued together along {\em slits}.  
We call this a partition of the surface. The main object of study in both parts of the theorem concerns the  summability of the areas of the changes of the partitions, expressed in terms  (\ref{ieq:sumx}) of the summability of the cross-product of the vectors of the slits.  

\subsection{Sketch of dimension $0$ case}
The starting point for the proof of Hausdorff dimension $0$ in the case that
\begin{equation}\label{PM:div}
  \sum_k \frac{\log\log q_{k+1}}{q_k} = \infty
\end{equation}  
is   Theorem~\ref{thm:CE} from \cite{CE}. 
That theorem asserts that 
  to each nonergodic direction $\theta\in\NE(P_\lambda)$ there is an associated sequence 
  of slits $\{w_j\}$ and loops $\{v_j\}$ whose directions converge to $\theta$ and satisfy the summability condition (\ref{ieq:sumx}).  
The natural language to describe the manner by which a sequence of vectors is associated 
  to a nonergodic direction is within the framework of $Z$-expansions.\footnote{This is 
  more of a convenience than an essential tool.}  (See \S\ref{s:LVdir}.)  
Here, $Z$ denotes a closed discrete subset of $\R^2$ satisfying some mild restrictions and 
  in the case when $Z$ is the set of primitive vectors in $\Z^2$ this notion reduces to 
  continued fraction expansions.  
We also have the notion of  {\em Liouville direction} (relative to $Z$) which intuitively refers to a direction 
  that is extremely well approximated by the directions of vectors in $Z$.  Under fairly general 
  assumptions, which hold for example if $Z$ is a set of holonomies of saddle connections on 
  a translation surface, the set of Liouville directions has Hausdorff dimension zero.  (Corollary~\ref{cor:LVdir}) 
The proof of Hausdorff dimension $0$  then reduces to showing that if $\lambda$ satisfies (\ref{PM:div}), then 
  every minimal nonergodic direction is Liouville with respect to the $Z$ expansion.  
This is stated as Lemma~\ref{lem:NE->LV}.

For the proof of  Lemma~\ref{lem:NE->LV}
the key ingedient  is Lemma~\ref{lem:min:area}, which gives a {\em lower bound} on cross-products.  
It is based on the fact that $\frac{p_k+mq_k}{nq_k}$ will be 
  an extremely good approximation to $\frac{\lambda+m}{n}$ provided the interval $[q_k,q_{k+1}]$ 
  is large enough and also contains $n$ not too close to $q_{k+1}$.  (See Lemma~\ref{lem:LC}.) 
  This idea is motivated by 
  the elementary fact that for any pair of vectors $w=(\tfrac{p}{q}+m,n)$ and $v=(m',n')$ where 
  $m,n,m',n',p,q\in\Z$ with $q>0$ we have $$|w\times v|=\frac{|(p+m)n'-m'nq|}{q}\ge\frac{1}{q}$$ 
  unless $v,w$ are parallel to each other, in which case the cross-product vanishes.  

We apply  Lemma~\ref{lem:min:area} to the   
sequence $\{w_j\}$ associated by Theorem~\ref{thm:CE} to a minimal nonergodic direction 
  $\theta$.  If one  assumes, by contradiction, that $\theta$ is {\em not} Liouville with respect to the $Z$-expansion,  then Lemma~\ref{lem:min:area} implies that 
  $$|w_j\times v_j|\ge\frac{1}{2q_k}$$ whenever $|w_j|$ falls in a large interval $[q_k,q_{k+1}]$.  Moreover, 
  the number of such slits is at least  a fixed constant times $\log\log q_{k+1}$.  Thus the sum of the cross-products would be at least $$\sum\frac{\log\log q_{k+1}}{q_k},$$ the sum over those $k$ for which $[q_k,q_{k+1}]$ is large.  
Since (\ref{PM:div}) still holds if the sum is restricted to those $k$,
the summable cross-products condition (\ref{ieq:sumx}) would be contradicted.  This will then show that $\theta$ is Liouville and we will conclude that $\Hdim\NE(P_\lambda)=0$.  

\subsection{Sketch of dimension $1/2$ case}
The starting point for the dimension $\tfrac12$ argument is Theorem~\ref{thm:sumx}, which is 
  the specialization of a result from \cite{MS} to the case of $(X,\omega)$ that says the summability 
  condition (\ref{ieq:sumx}) is sufficient to guarantee that the limiting direction of a sequence of slit 
  directions is a nonergodic direction.  

One proceeds to construct  
  a Cantor set of nonergodic directions arising as a limit of directions of slits on the torus.  
Aspects of this construction were already carried out in \cite{Ch1} in the case that  
 $\lambda$ is Diophantine. 

For $r>1$,  let $F(r)$ be the set of limiting directions obtained from sequences $\{w_j\}$ satisfying 
  $|w_{j+1}|\approx|w_j|^r$. 
It was shown in \cite{Ch1}, under the assumption of Diophantine $\lambda$, that one can make  
  the series in (\ref{ieq:sumx}) 
  be  dominated by a geometric series of ratio $1/r$, and then $\Hdim F(r)\ge\frac{1}{1+r}$.  
The lower bound $\tfrac12$ then follows by taking the limit as $r$ tends to one.  

The strategy of bounding cross-products using a geometric series fails if only the weaker Diophantine 
  condition (\ref{PM:conv}) is assumed.  In fact, in the large gaps $[q_k.q_{k+1}]$, as we have indicated, 
  the cross-product is bounded below by $\frac{1}{2q_k}$.  So if the gaps are large, (where the notion 
  of ``large'' is to be made precise later)  then there are many terms with cross-products bounded below 
  by $\frac{1}{2q_k}$ and these terms would eventually become larger than the terms in the geometric series. 

This suggests modifying the strategy in \cite{Ch1} by replacing the geometric series used to dominate the 
  series in (\ref{ieq:sumx}) with a series whose terms $\delta_j$ are $O(1/q_k)$ if $|w_j|$ lies in a large interval 
  $[q_k,q_{k+1}]$ and are otherwise decreasing like a geometric series of ratio $1/r$ for $j$ such that $|w_j|$ 
  lies between successive large intervals.  
The number of slits in $[q_k,q_{k+1}]$ is $O(\log_r\log q_{k+1})$ so that $\sum\delta_j$ 
  restricted to those $j$ for which $|w_j|$ lies in a large interval $[q_k,q_{k+1}]$ is bounded using 
  the assumption (\ref{PM:conv}). 
The sum of the remaining terms is bounded by the sum of a geometric series times $\sum_k\frac{1}{q_k}$.  
  This latter sum is finite.  The finiteness then of $\sum\delta_j$ and therefore (\ref{ieq:sumx}) ensures that 
  the resulting set $F(r)\subset\NE(P_\lambda)$.  

Following \cite{Ch1}, we seek to build a tree of slits so that by associating intervals about the direction of 
  each slit in the tree, we can give $F(r)$ the structure of a Cantor set to which standard techniques can be 
  used to give lower estimates on Hausdorff dimension.   These techniques require certain ``local estimates'' 
  (expressed in terms of lower bounds on the number of subintervals and the size of gaps between them) 
  hold at each stage of the construction.  
In \S\ref{s:Cantor}, we express these local estimates in terms of the parameters $r$ and $\delta_j$.  

For slits $w$ whose lengths lie in a "small" interval $[q_k,q_{k+1}]$ we repeat the construction given in 
  \cite{Ch1} to construct "children" slits from "parent" slits.  This is carried out in \S\ref{s:Diophantine}.  
In the current situation we have to combine that construction with a new one to deal with slits lengths 
  that lie between consecutive $q_k,q_{k+1}$ with large ratio.  We call this the "Liouville" part of $\lambda$.  
The construction of new slits from old ones in that case is carried out in \S\ref{s:Liouville}.  
 
The construction of the tree of slits and the precise definition of the terms $\delta_j$ are given in \S\ref{s:Init} 
  and \S\ref{s:Tree}.  These sections are the most technical part of the paper.  The main task is to ensure that 
  the recursive procedure for constructing the tree of slits can be continued indefinitely while at the same time 
  ensuring the required local estimates are satisfied in the case of our two constructions.  

Finally, in \S\ref{s:Lower}, we verify that the series $\sum\delta_j$ is convergent and that the lower bound 
  on $\Hdim F(r)$ can be made arbitrarily close to $\tfrac12$ by choosing the parameter $r$ sufficiently 
  close to one.

\subsection{Divergent geodesics}
Finally we record the following by-product of our investigation. 
Associated to any translation surface (or more generally a holomorphic 
quadratic differential) is a Teichm\"uller geodesic.  For each $t$ the Riemann 
surface $X_t$ along the geodesic is found by expanding along horizontal 
lines by a factor of $e^t$ and contracting along vertical lines by $e^t$.  
It is known (see \cite{Ma2}) that if the vertical foliation of the 
quadratic differential is nonergodic, then the associated Teichm\"uller 
geodesic is \emph{divergent}, i.e. it eventually leaves every compact 
subset of the stratum.\footnote{In \cite{Ma2}, a stronger assertion 
was proved, namely the \emph{projection} of the Teichm\"uller geodesic 
to the moduli space of Riemann surfaces is also divergent.}
The converse is however false.  
There are divergent geodesics for which the vertical foliation is uniquely ergodic.  
In fact,  we have 
\begin{theorem}\label{thm:divergent}
Let $\DIV(P_\lambda)$ denote the set of divergent directions in $P_\lambda$, 
  i.e. directions for which the associated Teichm\"uller geodesic leaves 
  every compact subset of the stratum.\footnote{Theorem~\ref{thm:divergent} 
  remains valid if $\DIV(P_\lambda)$ is interpreted as the set of directions that 
  are divergent in the sense described in the previous footnote.}  
Then $\Hdim \DIV(P_\lambda)=0$ or $\tfrac12$, with the latter case occurring 
  if and only if $\lambda$ is irrational.
\end{theorem}

The authors would like to thank Emanuel Nipper and the referee 
for many helpful comments.  

\section{Loops, slits, and summable cross-products}\label{s:NEdir}
In this section, we establish notation, study partitions of the surface 
  associated to $P_\lambda$, and recall the summable cross-products 
  condition (\ref{ieq:sumx}) for detecting nonergodic directions.  

Let $(T;z_0,z_1)$ denote the standard flat torus with two marked points.  
A \emph{saddle connection} on $T$ is a straight line that starts and ends 
  in $\{z_0,z_1\}$ without meeting either point in its interior.  
By a \emph{slit} we mean a saddle connection that joins $z_0$ and $z_1$, 
  while a \emph{loop} is a saddle connection that joins either one of these 
  points to itself.  

Holonomies of saddle connections will always be represented as a pair 
  of real numbers.  In particular, $$\hol(\gamma_0)=(\lambda,0)$$ 
  where $\gamma_0$ is the horizontal slit joining $z_0$ to $z_1$.  
The set of holonomies of loops is given by 
  $$V_0=\{(p,q)\in\Z^2: \gcd(p,q)=1\}.$$  
Since $\lambda$ is irrational, the set of holonomies of slits is given by 
  $$V_1=V_1^+\cup\left(-V_1^+\right)$$ where
  $$V_1^+=\{(\lambda+m,n): m,n\in\Z^2, n>0\}\cup\{(\lambda,0)\}.$$ 

Note that $V_0$ and $V_1$ are disjoint and that $V_1$ is in one-to-one 
  correspondence with the set of oriented slits.  
When we speak of ``the slit $w$ ...'' we shall always mean the slit whose 
  holonomy is $w$, while $w\in V_1^+$ specifies that the orientation is 
  meant to be from $z_0$ to $z_1$.  
Also, each $v\in V_0$ corresponds to a pair of loops, one based at each 
  branch point.  The pair of cylinders in $T$ bounded by these loops will 
  be denoted by $C^1_v,C^2_v$.  The core curves of these cylinders also 
  have $v$ as their holonomy.  

\begin{definition}
Each slit $\gamma$ has two lifts in $(X,\omega)$ whose union is a 
  simple closed curve.  We say $\gamma$ is \emph{separating} if this 
  curve separates $X$ into a pair of tori interchanged by the involution 
  of the double cover.\footnote{This involution, which fixes each branch 
  point, should not be confused with the hyperelliptic involution that 
  interchanges the branch points and maps each slit torus to itself.}  
We denote the slit tori by $T^1_w,T^2_w$ where $w=\hol(\gamma)$.  
\end{definition}

\begin{lemma}(\cite{Ch1})
A slit $w$ is separating if and only if $w=(\lambda+m,n)$ 
  for some even integers $m,n$.  
\end{lemma}
The collection of separating slits have holonomies given by 
  $$V_2=V_2^+\cup\left(-V_2^+\right)$$ where
  $$V_2^+=\{(\lambda+2m,2n): m,n\in\Z^2, n>0\}\cup\{(\lambda,0)\}.$$ 

The cross-product formula from vector calculus expresses the area of the 
  parallelogram spanned by $u$ and $v$ as 
  $$|u\times v|=\|u\| \|v\|\sin\theta$$ 
  where $\times$ denote the standard skew-symmetric bilinear form on $\R^2$, 
  $\|\cdot\|$ the Euclidean norm, and $\theta$ the angle between $u$ and $v$.  
It will be convenient to introduce the following.  

\begin{notation}
The distance between the directions of $u,v\in\R\times\R_{>0}$, 
  denoted by $\angle uv$, will be measured with respect to inverse 
  slope coordinates.  That is, $\angle uv$ is the absolute value of 
  the difference between the reciprocals of their slopes.  
We have the folllowing analog of the cross-product formula 
  $$|u\times v| = |u|~|v|\angle uv$$ 
  where $|\cdot|$ denotes the absolute value of the $y$-coordinate.  
\end{notation}  

\begin{remark}
For our purposes, the vectors we consider will always have directions 
  close to some fixed direction and nothing essential is lost if one 
  chooses to think of $|v|$ as the \emph{length} of the vector $v$ 
  (or to think of $\angle uv$ as the angle between the vectors) for 
  these notions differ by a ratio that is nearly constant.  
In fact, the notations $|v|$ and $\angle uv$ are intended to remind 
  the reader of Euclidean lengths and angles, and in the discussions 
  we shall sometimes refer to them as such.  
These nonstandard notions are particularly convenient in calculations 
  as they allows us to avoid trivial approximations involving square 
  roots and the sine function that would otherwise be unavoidable 
  had we instead insisted on the Euclidean notions.  
As will become clear later, the benefits of the nonstandard notions 
  will far outweigh the potential risks of confusion.  
\end{remark}

\begin{lemma}\label{lem:disjoint}
Let $C^1_v,C^2_v$ be the cylinders in $T$ determined by $v\in V_0$.  
A slit $w$ is contained in one of the cylinders $C^i_v$ if and only if 
  $|w\times v|<1$.  
\end{lemma}
\begin{proof}
To prove necessity, we note that the area of the cylinder containing 
  the slit is $|w\times v|$, which is $<1$ since the complement has 
  positive area.  
For sufficiency, let us first rotate the surface so that $v$ is horizontal.  
If the slit were not contained in one of the cylinders, then the vertical 
  component of $w$ is a (strictly) positive linear combination of the 
  heights $h_1,h_2$ of the rotated cylinders.  
However, the vertical component is given by 
  $$\frac{|w\times v|}{\|v\|} < \frac{1}{\|v\|} = h_1+h_2$$ 
  which is absurd.  
\end{proof}

\begin{definition}
Let $w,w'\in V_1$ and $v\in V_0$.  
We shall say $w$ and $w'$ are ``related by a Dehn twist about $v$'' if 
  they are contained in the same cylinder determined by $v$.  
If both lie in $V_1^+$ (or both in $-V_1^+$) then their holonomies are 
  related by $w'=w+bv$ for some $b\in\Z$.  In this case, we refer to 
  $|b|$ as the \emph{order} of the Dehn twist.  
\end{definition}

\begin{lemma}\label{lem:twist}
Let $w,w'\in V_1$ and $v\in V_0$.  
If $|w\times v|+|w'\times v|<1$ then $w$ and $w'$ are related by 
  a Dehn twist about $v$.  
\end{lemma}
\begin{proof}
Lemma~\ref{lem:disjoint} implies each of $w$ and $w'$ is contained 
  in one of the cylinders $C^1_v$ and $C^2_v$ determined by $v$.  
If they belong to different cylinders, then the sum of the areas of 
  the cylinders would be less than one, which is impossible.  
Hence, $w$ and $w'$ lie in the same cylinder and, therefore, they 
  are related by a Dehn twist about $v$.  
\end{proof}

Suppose $w,w'$ are a pair of separating slits.
Then we may measure the change in the partitions they determine by 
  $$\chi(w,w'):=\area(T^1_w\Delta T^1_{w'}).$$  
There is an ambiguity in this definition arising from the fact that 
  we have not tried to distinguish between $T^1_w$ and $T^2_w$.  
Let us agree to always take the smaller of the two possibilities, which 
  is at most one as their sum represents the area of $(X,\omega)$.  

\begin{lemma}\label{lem:area}
If $w,w'$ are separating slits related by a Dehn twist about $v$ then 
  $$\chi(w,w') = |w'\times v| = |w\times v|.$$  
\end{lemma}
\begin{proof}
Let $C$ be the cylinder that contains both slits and let $b>0$ be the 
  order of the Dehn twist relating them.  Note that $b$ is even.  
The slits cross each other, each subdividing the other into $b$ 
  segments of equal length.  
The symmetric difference between the partitions is a finite union of 
  parallelograms bounded by the lifts of $w$ and $w'$.  
There are $b$ parallelograms, each having area 
  $\frac{1}{b^2}|w\times w'|$ and since $w'=w\pm bv$, we have 
  $|w\times w'|=b|w\times v|=b|w'\times v|$, giving the lemma.  
\end{proof}

Each separating slit determines a partition of $(X,\omega)$ into 
  a pair of slit tori of equal area.  
The next theorem explains how nonergodic directions arise 
  as certain limits of such partitions.  
It is a special case, adapted to branched double covers of tori, 
  of a more general condition developed in \cite{MS} that applies 
  to arbitrary translation surfaces and quadratic differentials.   
We will use it in \S\ref{s:Lower} to identify large subsets of 
  $\NE(P_\lambda)$.

\begin{theorem}\label{thm:sumx}
Let $\{w_j\}$ be a sequence of separating slits with increasing 
  lengths $|w_j|$ and suppose that every consecutive pair of slits 
  $w_j$ and $w_{j+1}$ are related by a Dehn twist about some $v_j$ 
  such that 
\begin{equation}\label{ieq:sumx}
  \sum_j |w_j\times v_j| < \infty.  
\end{equation}
Then the inverse slopes of $w_j$ converge to some $\theta$ and 
  this limiting direction belongs to $\NE(P_\lambda)$.  
\end{theorem}
\begin{proof}
Since $|w_{j+1}|>|w_j|$, we have $|v_j|\ge1$ so that 
\begin{align*}
  \angle w_jw_{j+1} \le \frac{|w_j\times v_j|}{|w_j||v_j|} 
        + \frac{|v_j\times w_{j+1}|}{|v_j||w_{j+1}|} 
    \le \frac{2|w_j\times v_j|}{|w_j|}
\end{align*}
  from which the existence of the limit $\theta$ follows. 
Let $\mu$ be the normalised area measure on $(X,\omega)$ and 
  let $h_j$ be the component of $w_j$ orthogonal to $w_\infty=(\theta,1)$.  
Theorem~2.1 in \cite{MS} asserts that $\theta$ is a nonergodic direction 
  if the following conditions hold: 
\begin{enumerate}
  \item[(i)] $\lim h_j=0$, 
  \item[(ii)] $0<c<\mu(T^1_{w_j})<c'<1$ for some constants $c,c'$, and 
  \item[(iii)] $\sum \chi(w_j,w_{j+1}) <\infty$.  
\end{enumerate}
Since $\mu(T^1_{w_j})=\tfrac12$, (ii) is clear, while (iii) is a consequence 
  of (\ref{ieq:sumx}), by Lemma~\ref{lem:area}.  
It remains to verify (i), but this follows easily from 
  $$\angle w_jw_\infty \le \sum_{i\ge j} \angle w_iw_{i+1} 
        = \sum_{i\ge j}\frac{2|w_i\times v_i|}{|w_i||w_{i+1}|} 
        \le \sum_{i\ge j}\frac{2|w_i\times v_i|}{|w_j|}$$  
  since then $h_j\le|w_j|\angle w_jw_\infty\le\sum_{i\ge j}2|w_i\times v_i|$ 
  so that $h_j\to0$, by (\ref{ieq:sumx}).  
\end{proof}


The converse to Theorem~\ref{thm:sumx} also holds.  
That is, to each nonergodic direction $\theta$ one can 
  associate a sequence of slits $(w_j)$ whose directions 
  converge to $\theta$ and such that all the hypotheses 
  of Theorem~\ref{thm:sumx} hold.  
The definition of this sequence will be explained next.

\section{$Z$-expansions, Liouville directions}\label{s:LVdir}
In this section we introduce $Z$-expansions and use them to 
  define the notion of a Liouville direction relative to a closed 
  discrete subset $Z\subset\R^2$.  
Under fairly general assumptions on $Z$, the set of Liouville 
  directions is shown to have Hausdorff dimension zero.  

\begin{notation}
Given an inverse slope $\theta$ and $v=(p,q)\in\R^2$ we define 
  $$\hor_\theta(v)=|q\theta-p|$$ which we shall refer to as the 
  ``horizontal component'' of $v$ in the direction $\theta$.  
It represents the absolute value of the $x$-coordinate of the 
  vector $h_\theta v$ where $h_\theta=\pmat{1&-\theta\\0&1}$ 
  is the horizontal shear that sends the direction of $\theta$ 
  to the vertical.  
\end{notation}

\begin{definition}
Let $Z$ be a closed discrete subset of $\R^2$ and $\theta$ an 
  inverse slope.  A $Z$-\emph{convergent} of $\theta$ is any 
  vector $v\in Z$ that minimizes the expression $\hor_\theta(u)$ 
  among all vectors $u\in Z$ with $|u|\le|v|$.  
Recall that $|v|$ is the absolute value of the $y$-coordinate.  
We call it the \emph{height} of $v$.\footnote{The height of a 
  rational is the smallest positive integer that multiplies it into 
  the integers.  A rational represented in lowest terms by $p/q$ 
  can be identified with $v=(p,q)\in\Z^2$, so that the height of 
  the vector $v$ coincides with the height of the rational.}  
Thus, $Z$-convergents are those vectors in $Z$ that minimize 
  horizontal components among all vectors in $Z$ of equal of 
  lesser height.  
The $Z$-\emph{expansion} of $\theta$ is defined to be the 
  sequence of $Z$-convergents ordered by increasing height.  
If two or more $Z$-convergents have the same height 
  we choose one and ignore the others.  
\end{definition}

Note that by definition the sequence of heights of $Z$-expansion 
  is \emph{strictly} increasing and, as a consequence, the sequence 
  of horizontal components is strictly decreasing--if $|v|<|v'|$ then 
  $\hor_\theta(v)$ must be greater than $\hor_\theta(v)$, for otherwise 
  $v'$ would not qualify as a $Z$-convergent.  

In the case when $Z$ is the set of primitive vectors in $\Z^2$, 
  i.e. $Z=V_0$, the notion of a $Z$-convergent reduces to the 
  notion from continued fraction theory.  
That is, $v=(p,q)$ is a $Z$-convergent of $\theta$ if and only if 
  $p/q$ is a convergent of $\theta$ in the usual sense.\footnote{
  There is a trivial exception in the case when $\theta$ has fractional 
  part strictly between $\tfrac12$ and $1$: the integer part of $\theta$ 
  is the \emph{zeroth order} convergent of $\theta$ in the usual sense, 
  but nevertheless fails to be a $Z$-convergent.}
A generalisation to higher dimensions (where $Z$ is the set of 
  primitive vectors in $\Z^n$ for $n>2$) is given in \cite{Ch3}.  

Obviously, we should always assume $Z$ does not contain the 
  origin, for otherwise the zero vector is the only convergent, 
  independent of $\theta$.  
Let us also assume that $Z$ contains some nonzero vector on 
  the $x$-axis, for this ensures that the heights of 
$Z$-expansions are well-ordered.  
Indeed if $(x,0)$ is a $Z$-convergent, then all $Z$-convergents 
  lie in an infinite parallel strip of width $2x$ about the direction 
  of $\theta$.  Since the set of $Z$-convergents forms a closed 
  discrete subset of this strip, there is no accumulation point.  
Hence, if there are infinitely many $Z$-convergents, their 
  heights increase towards infinity.  

One last assumption we shall impose is the finiteness of the 
  ``Minkowski'' constant: 
\begin{equation}\label{Minkowski}
   \mu(Z) := \frac{1}{4}\sup_K \area(K) < \infty 
\end{equation}
  where the supremum is taken over all bounded, $0$-symmetric 
  convex regions disjoint from $Z$.  
Any direction which is not the direction of a vector in $Z$ will be 
  called \emph{minimal} (relative to $Z$).  
\begin{lemma}\label{lem:Minkowski}
Assume (\ref{Minkowski}) and that $Z$ contains a non-zero vector 
  on the $x$-axis.  Then the $Z$-expansion of a direction with 
  inverse slope $\theta$ is infinite if and only if $\theta$ is minimal.  
\end{lemma}
\begin{proof}
If the $Z$-expansion is finite, take the last convergent.  If it does 
  not lie in the direction of $\theta$, then there is an
 infinite parallel strip 
  containing the origin with one side the direction of $\theta$ 
  containing no points of $Z$, but this is ruled out by (\ref{Minkowski}).  
Hence, its direction is $\theta$, so $\theta$ is not minimal.  
Conversely, if $\theta$ is not minimal, then there is a vector in $Z$ in the 
  direction of $\theta$ and it is necessarily a convergent and no other 
  convergent can beat it, so it is the last one in the $Z$-expansion.  
There is also a first convergent; it lies on the $x$-axis.  
Let $x$ the horizontal component of the first convergent and 
  $y$ the height of the last convergent.  The compact region 
\begin{equation}\label{def:P(x,y)}
  P_\theta(x,y) = \{ v\in\R^2 : \hor_\theta(v)\le x, |v|\le y \}
\end{equation}
  contains all the $Z$-convergents.  
Since $Z$ is closed, it is compact; by discreteness, it is finite.  
\end{proof}

Note that the $Z$-expansion are defined for all directions except 
  the horizontal.  
In the sequel, we shall always assume the hypotheses of 
  Lemma~\ref{lem:Minkowski} remain in force.  

\begin{notation}
If $\theta$ is an inverse slope and $u$ a non-horizontal vector 
  then we shall often write $\angle u\theta$ for the absolute 
  difference between the directions.  That is, 
  $$\angle u\theta= \angle uv = \frac{|u\times v|}{|u||v|}$$  
  for any vector $v$ whose inverse slope is $\theta$.  
Similarly, the notation $|u\times\theta|$ will be used to mean 
  $$|u\times\theta|=\frac{|u\times v|}{|v|} = |u\times v_\theta|$$  
  where $v_\theta=(\theta,1)$.  
\end{notation}

\begin{theorem}\label{thm:angle}
The sequence of $Z$-convergents of $\theta$ satisfies\footnote{The 
  notation $\angle v\theta$, as in (\ref{ieq:angle}), means $\angle vw$ 
  for any $w$ whose inverse slope is $\theta$.}
\begin{equation}\label{ieq:angle}
  \frac{|v_k\times v_{k+1}|}{2|v_k||v_{k+1}|} < 
  \angle v_k\theta \le \frac{\mu(Z)}{|v_k||v_{k+1}|}.  
\end{equation}
\end{theorem}
\begin{proof}
Consider the parallelogram $P=P(x_k,y_{k+1})$ defined by (\ref{def:P(x,y)}) 
  where $x_k=\hor_\theta(v_k)$ and $y_{k+1}=|v_{k+1}|$.  
The base is $2|v_k\times\theta|$ and the height is $2|v_{k+1}|$.  
By definition of $v_{k+1}$, the interior of $P$ is disjoint from $Z$ so that 
  (\ref{Minkowski}) implies $$|v_k\times\theta||v_{k+1}|\le\mu(Z)$$
  giving the right hand inequality in (\ref{ieq:angle}).  
Since $$\angle v_{k+1}\theta = \frac{|v_{k+1}\times\theta|}{|v_{k+1}|} 
  < \frac{|v_k\times\theta|}{|v_k|} = \angle v_k\theta$$  
  we have $\angle v_kv_{k+1} < 2 \angle v_k\theta$, giving the left hand 
  inequality in (\ref{ieq:angle}).  
\end{proof}

\subsection{Liouville directions}\label{ss:LVdir}
Recall that an irrational number is Diophantine iff the sequence 
  of denominators of its convergents satisfies $q_{k+1}=O(q_k^N)$ 
  for some $N$.  Otherwise, it is Liouville.  
This motivates our next definition.  
\begin{definition}
We say a minimal direction is \emph{Diophantine} relative to $Z$ 
  if its $Z$-expansion satisfies 
\begin{equation}\label{def:D_N}
  |v_{k+1}|=O\left(|v_k|^N\right)
\end{equation}
  for some $N$.  Otherwise, it is \emph{Liouville} relative to $Z$.  
\end{definition}

Note that we have a trichotomy: every direction is either Diophantine, 
  Liouville or not minimal, relative to $Z$.  

\begin{definition}
We say $Z$ has \emph{polynomial growth} of rate (at most) $d$ if 
  $$\#(Z\cap B_R) = O(R^d)$$ 
  where $B_R$ denotes the ball of radius $R$ about the origin.  
\end{definition}  

\begin{lemma}\label{lem:E_r}
Let $E_r$ be the set of (inverse slopes of) directions $\theta$ whose 
  $Z$-expansions satisfy $$|v_{k+1}|>|v_k|^r$$ for infinitely many $k$.  
If $Z$ has polynomial growth of rate $d$, then 
  $$\Hdim E_r \le \frac{d}{1+r}.$$  
\end{lemma}
\begin{proof}
It is enough to bound the Hausdorff dimension of the set 
  $E'_r=E_r\cap[a,a+1]$ for some arbitrary but fixed $a\in\R$.  
Let $Z_k$ be the set of $v\in Z$ that arise as $Z$-convergents 
  of some direction whose inverse slope lies in $[a,a+1]$ and 
  such that $$2^k\le|v|<2^{k+1}.$$  
Then $Z_k$ is contained in some ball of radius $2^kR_0$ where 
  $R_0$ is a constant depending only on $a$.  
Let $I(v)$ be the closed interval of length $\frac{2\mu(Z)}{|v|^{1+r}}$ 
  centered about the inverse slope of $v$.  
Then Theorem~\ref{thm:angle} implies every $\theta\in E'_r$ is 
  contained in $I(v)$ for infinitely many $v\in\bigcup_k Z_k$.  
For any $k_0$  let 
  $$Z_{k_0}'=\bigcup_{k\ge k_0} Z_k.$$  
Then given $\eps>0$ we can choose $k_0$ large enough 
  so that $\{I(v):v\in Z_{k_0}'\}$ is an $\eps$-cover of $E'_r$.  
Since the number of elements in $Z_k$ is bounded by 
  $$\# Z_k \le CR_0^d2^{kd}$$ for some $C>0$ we have 
  $$\sum_{v\in Z'_{k_0}} |I(v)|^s \le \sum_{k\ge k_0} 
         \frac{2^s\mu(Z)^sCR_0^d2^{kd}}{2^{k(1+r)s}}$$  
  so that the $s$-dimensional Hausdorff measure is finite 
  for any $s>\frac{d}{1+r}$.  
This shows $\Hdim E'_r\le\frac{d}{1+r}$, from which the 
  lemma follows.  
\end{proof}

By \cite{Ma1} (see also \cite{EM}, \cite{Vo}) the set of holonomies 
  of saddle connections on any translation surface satisfies a 
  quadratic growth rate.  
\begin{corollary}\label{cor:LVdir}
The set of Liouville directions relative to the set of holonomies 
  of saddle connections on a translation surface has Hausdorff 
  dimension zero.  
\end{corollary}

\section{Hausdorff dimension $0$}\label{s:Upper}
In this section we assume the denominators of the convergents 
  of $\lambda$ satisfy (\ref{PM:div}) and set $$Z=V_0\cup V_2.$$  
(Recall the sets $V_0$ and $V_2$ were defined in \S\ref{s:NEdir}.)  

We shall need the following characterisation of nonergodic 
  directions in terms of $Z$-expansions.  
\begin{theorem}\label{thm:CE}
(\cite{CE}) 
Let $\theta$ be a minimal\footnote{This implies it will also be 
  a minimal direction relative to $Z$.} direction in $P_\lambda$.  
Then $\theta$ is nonergodic if and only if its $Z$-expansion 
  is eventually alternating between loops and separating slits 
  $$\ldots, v_{j-1}, w_j, v_j, w_{j+1}, \ldots$$ 
  and satisfies the summable cross-products condition (\ref{ieq:sumx}).  
\end{theorem}

Our goal is to show that $\Hdim \NE(P_\lambda)=0$ under the 
  assumption (\ref{PM:div}).  
By Corollary~\ref{cor:LVdir}, it is enough to show that every 
  minimal nonergodic direction is Liouville relative to $Z$.  

Note that the sufficiency in Theorem~\ref{thm:CE} follows from 
  Theorem~\ref{thm:sumx} since the heights of $Z$-convergents 
  increase and as soon as $|w_{j+1}\times v_j|=|w_j\times v_j|<\frac12$ 
  then $w_j$ and $w_{j+1}$ are related by a Dehn twist about $v_j$, 
  by Lemma~\ref{lem:twist}.  
The main point of Theorem~\ref{thm:CE} is that the converse 
  also holds.  

Observe that our main task has been reduced to a question about 
  the set of possible limits for the directions of certain sequences 
  of vectors in $Z$.  

In the sequel we shall need the following two standard facts from 
  the theory of continued fractions.  
\begin{theorem}\label{thm:CF1}([Kh, Thm.~9 and 13])
The sequence of convergents of a real number $\theta$ satisfies 
\begin{equation}\label{ieq:CF1}
  \frac{1}{q_k(q_k+q_{k+1})} < \left|\theta-\frac{p_k}{q_k}\right| 
                           \le \frac{1}{q_kq_{k+1}}.
\end{equation}
\end{theorem}

\begin{theorem}\label{thm:CF2}([Kh, Thm.~19])
If a reduced fraction satisfies 
\begin{equation}\label{ieq:CF2}
  \left|\theta-\frac{p}{q}\right|<\frac{1}{2q^2}  
\end{equation}
  then it is a convergent of $\theta$. 
\end{theorem}

\subsection{Liouville convergents}\label{ss:LC}
The next lemma shows that convergents of $\lambda$ with $q_{k+1}\gg q_k$ 
  give rise to convergents of $\frac{\lambda+m}{n}$.  
\begin{lemma}\label{lem:LC}
Let $w=(\lambda+m,n)$ be a slit and $\frac{p_k}{q_k}$ a convergent of $\lambda$ 
  such that 
\begin{equation}\label{sc:LC}
  |w| = n < \frac{q_{k+1}}{2q_k}.  
\end{equation}
Let $\frac{p}{q}$ denote the fraction $\frac{p_k+mq_k}{nq_k}$ in lowest terms.  
Then $\frac{p}{q}$ is a convergent of $\frac{\lambda+m}{n}$ and 
  its height satisfies $q_k\le q\le|w|q_k$.  
Furthermore, the height $q'$ of the next convergent of $\frac{\lambda+m}{n}$ 
  is larger than $\frac{q_{k+1}}{2}$.  
\end{lemma}
\begin{proof}
Using the right hand side of (\ref{ieq:CF1}) and (\ref{sc:LC}) we get 
\begin{equation}\label{ieq:LC}
  \left|\frac{\lambda+m}{n} - \frac{p_k+mq_k}{nq_k}\right| < \frac{1}{|w|q_kq_{k+1}} 
     < \frac{1}{2n^2q_k^2}
\end{equation}
  which implies that $\frac{p}{q}$ is a convergent of $\frac{\lambda+m}{n}$.  
Clearly, $q\le|n|q_k=|w|q_k$ and since $\gcd(p_k,q_k)=1$, $n$ is divisible by 
  $\gcd(p_k+mq_k,nq_k)$ so that $q\ge q_k$.  
Let $q'$ be the height of the next convergent of $\frac{\lambda+m}{n}$.  
From the first inequalities in (\ref{ieq:CF1}) and in (\ref{ieq:LC}) we get 
  $$\frac{1}{2qq'} < \left|\frac{\lambda+m}{n}-\frac{p}{q}\right| < \frac{1}{|w|q_kq_{k+1}}$$  
  so that $$q'>\frac{|w|q_kq_{k+1}}{2q}\ge\frac{q_{k+1}}{2}.$$  
\end{proof}

\begin{definition}
When the conclusion of Lemma~\ref{lem:LC} holds, we refer to $\frac{p}{q}$ (or 
  the vector $v=(p,q)$) as the \emph{Liouville convergent} of $w$ indexed by $k$.  
(We shall often blur the distinction between the rational $\frac{p}{q}$ and the vector $v$.)  
\end{definition}

The terminology of Liouville convergent is justified by the sequel both in the 
dimension $0$ result and in the dimension $1/2$ result.  In the next lemma 
we show that if $w',w$ have their lengths in a range defined by the convergents 
of $\lambda$ and are related by a twist about a loop $v$, then if $v$ is not the 
Liouville convergent of $w$, the area interchange determined by $w,w'$ will be large.  
If $v$ is the Liouville convergent, then the next slit after $w'$ will not be in the range.  
The summability condition on area exchanges will then imply that there cannot be too 
many slit lengths in the Liouville part of $\lambda$ (in the range where $q_{k+1}/q_k$ 
is large).  Consequently the lengths of the slits must grow quickly and we can find 
covers of the nonergodic set that allow us to prove Hausdorff dimension $0$ using 
Lemma~\ref{lem:E_r}.  In \S\ref{s:Liouville} we will use Liouville convergents to build 
new children slits out of parent slits.

\begin{lemma}\label{lem:min:area}
Let $w,w'$ be slits such that $w,w'$ are related by a Dehn twist about $v\in V_0$ and 
  $|w\times v|<\tfrac12$.   
Suppose further that $|w|<|w'|<\frac{q_{k+1}}{2q_k}$ and let $u$ be the Liouville 
  convergent of $w$ indexed by $k$.  Regarding $u$ as a vector, then either 
\begin{enumerate}
\item[(i)]  $v\neq u$ and 
\begin{equation}\label{ieq:min:area}
  |w\times v|  > \frac{1}{2q_k}, 
\end{equation}
  or 
\item[(ii)]  $v=u$ and for any $v'\in\Z^2\setminus\Z v$ satisfying 
  $|w'\times v'|<\tfrac12$ we have $|v'|>\frac{q_{k+1}}{4}$.  
\end{enumerate}
\end{lemma}
\begin{proof}
We have $w'=w+bv$ for some nonzero, even integer $b$, so that 
\begin{equation}\label{ieq:v<w'}
  |v|=\frac{|w'-w|}{|b|}\le\frac{|w'|+|w|}{2}<|w'|.  
\end{equation}
Let $\alpha'$ be the inverse slope of $w'$.  Let $v=(p,q)$. Then 
  $$\left|\alpha'-\frac{p}{q}\right| = \frac{|w'\times v|}{|w'||v|} 
      < \frac{|w\times v|}{|v|^2} < \frac{1}{2q^2}$$
  so that $\frac{p}{q}$ is a convergent of $\alpha'$, by (\ref{ieq:CF2}).  
Let $q'$ be the height of the next convergent of $\alpha'$.  
Then (\ref{ieq:CF1}) implies 
  $$\frac{1}{2qq'}<\left|\alpha'-\frac{p}{q}\right| < \frac{1}{qq'}$$ 
  so that 
  $$\frac{|w'|}{2|w\times v|} = \frac{|w'||v|}{2q|w'\times v|} < q' < \frac{|w'|}{|w\times v|}.$$  
The Liouville convergent $u=(m,n)$ cannot have its height $n<q$ because 
  Lemma~\ref{lem:LC} implies the height $n'$ of the next convergent of 
  $\alpha'$ is greater than $\frac{q_{k+1}}{2}>|w'|>|v|=q$, contradicting 
  the fact that $q$ is the height of a convergent of $\alpha'$, namely 
  $\frac{p}{q}$.  
Thus, $|u|\ge|v|$.  

In case (i), $|u|>|v|$ so that $|u|\ge q'>\frac{|w'|}{2|w\times v|}$.  
Since $|u|\le|w'|q_k$, the inequality (\ref{ieq:min:area}) follows.  

In case (ii), we have $q'=n'>\frac{q_{k+1}}{2}$, as noted earlier.  
Given $v'\in\Z^2\setminus\Z v$, we have 
\begin{align*}
  1\le|u\times v'|&=|u||v'|\left(\angle w'u + \frac{|w'\times v'|}{|w'||v'|}\right) \\
      &\le \frac{|v'|}{q'} + \frac{|v|}{2|w'|} < \frac{|v'|}{q'}+\frac{1}{2} 
\end{align*}
  from which it follows that $|v'|>\frac{q'}{2}>\frac{q_{k+1}}{4}$.  
\end{proof}

The Hausdorff dimension $0$ result now follows from 
\begin{lemma}\label{lem:NE->LV}
Assume $$\sum_k \frac{\log\log q_{k+1}}{q_k} = \infty$$ holds.  
Then any minimal $\theta\in\NE(P_\lambda)$ is Liouville relative to $Z$.  
\end{lemma}
\begin{proof}
Let $n_k>1$ be defined by $q_{k+1}=q_k^{n_k}$; in other words, 
  $$n_k = \log_{q_k}q_{k+1} = \frac{\log q_{k+1}}{\log q_k}.$$  
Note that since $q_k$ grows exponentially, we have 
  $$\sum_{n_k\le N} \frac{\log\log q_{k+1}}{q_k} 
      \le \sum_{n_k\le N} \frac{\log N+\log\log q_k}{q_k} < \infty$$
  for any $N>0$.  
Hence, (\ref{PM:div}) implies $n_k$ is unbounded; moreover, the series 
  in (\ref{PM:div}) diverges even if we restrict to terms with $n_k>N$.  

Let $\theta\in\NE(P_\lambda)$ be a minimal direction for the flow.  
Then it is minimal relative to $Z$ and by Theorem~\ref{thm:CE} 
  its $Z$-expansion eventually alternates $\dots,w_j,v_j,w_{j+1},\dots$ 
  between (separating) slits and loops such that (\ref{ieq:sumx}) holds.  
Let $J_k$ be the collection of indices $j$ such that 
  $$q_k \le |w_j| < |w_{j+1}| < |w_{j+2}|<q_k^{n_k-2} < \frac{q_{k+1}}{4q_k}.$$  
For any $j\in J_k$ we wish to prove that conclusion (i) of Lemma~\ref{lem:min:area} holds.  
Suppose by way of contradiction conclusion (ii) holds so that 
  $v_j$ is  the Liouville convergent of $w_j$ indexed by $k$.  
Setting $v'=v_{j+1}$ by conclusion (ii) we have $|w_{j+2}|>|v_{j+1}|>\frac{q_{k+1}}{4}$, 
  a contradiction.  Thus (i) holds and therefore $|w_j\times v_j|>\frac{1}{2q_k}$.  

Suppose $\theta$ is Diophantine relative to $Z$.  
Then there exists $N$ such that $|w_{j+1}|<|w_j|^N$ for all $j$.  
Hence, $|w_j|<|w_0|^{N^j}$ and since 
  $$\log_N\log_{|w_0|}q^a=\frac{1}{\log N}(\log a+\log\log q-\log|w_0|)$$ 
  we see that the number of $j$ such that $|w_j|$ lies in an interval of 
  the form $[q^a,q^b]$ is at least $\lfloor\log_N(b/a)\rfloor$.  
It follows that the number of elements in $J_k$ is at least 
  $$\log_N(n_k-2)-3 > \frac{\log n_k}{2\log N} 
         = \frac{\log\log q_{k+1}-\log\log q_k}{2\log N}$$ 
  provided $n_k>N_0$ for some $N_0$ depending only on $N$.  
Since $\sum\frac{\log\log q_k}{q_k}<\infty$ (as heights of convergents 
  grow exponentially) we have 
  $$\sum_{n_k>N_0}\sum_{j\in J_k} |w_j\times v_j| > 
              \sum_{n_k>N'}\frac{\log\log q_{k+1}-\log\log q_k}{2(\log N)q_k} = \infty$$ 
  which contradicts (\ref{ieq:sumx}).  
Hence, $\theta$ must be Liouville relative to $Z$, proving the lemma.  
\end{proof}

\section{Cantor set construction}\label{s:Cantor}
We begin the proof of the Hausdorff dimension $1/2$ result.  
To construct nonergodic directions, we use Theorem~\ref{thm:sumx}.  
The general idea is as follows.  
Starting with an initial slit $w_0$ we will construct a tree of slits.  
At level $j$ we will have a collection of slits of approximately 
  the same length.  
For each $w$ in this collection we wish to construct new slits 
  of level $j+1$ each having small cross-product with $w$.  
Depending on the relationship of the length of $w$ to the 
  continued fraction expansion of $\lambda$, as specified 
  precisely in \S\ref{s:Init}, the construction will be one of two types 
  that will be explained in \S\ref{s:Liouville} and \S\ref{s:Diophantine}.  

In this section, we associate to this tree of slits a Cantor set.  
For each $j$ we will define a set $F_j$ which is a disjoint union of intervals.  
The directions of each slit of level $j$ will lie in some interval in $F_j$ 
  and the intervals at level $j$ will be separated by gaps.  
The intervals of level $j+1$ will be nested in the intervals of level $j$.  
Each nonergodic direction corresponds to a nested intersection of these intervals.  

We shall assume the tree of slits satisfy certain assumptions, 
  to be verified later in \S\ref{s:Tree} and \S\ref{s:Lower}.  
These assumptions, expressed in terms of parameters $r>1$, $\delta_j>0$ 
  and $\rho_j>0$, ensure that certain lower bounds on the Hausdorff dimension 
  of the Cantor set will hold.  

\subsection{Local Hausdorff dimensions}\label{ss:Hdim}
To establish lower bounds for Hausdorff dimension we will use an 
  estimate of Falconer \cite{Fa} which we explain next.  
Let $$F=\bigcap_{j\ge0} F_j$$ where each $F_j$ is a finite disjoint union 
  of closed intervals and $F_{j+1}\subset F_j$ for all $j$.  
Suppose there are sequences $m_j\ge2$ and $\eps_j\searrow0$ 
  such that each interval of $F_j$ contains at least $m_j$ intervals 
  of $F_{j+1}$ and the smallest gap between any two intervals 
  of $F_{j+1}$ is at least $\eps_j$.  
(Note that $m_j\ge2$ implies there will always be at least one gap.)  
Then Falconer's lower bound estimate is 
\begin{equation*}\label{eq:Falc}
\Hdim F \ge \liminf_j \frac{\log (m_0\cdots m_j)}{-\log m_{j+1}\eps_{j+1}}.  
\end{equation*} 
If $\lim_{j\to\infty}m_j\eps_j=0$, as is necessarily the case if the 
  length of the longest interval in $F_j$ tends to zero as $j\to\infty$, 
  then $$\Hdim F \ge \liminf_j d_j$$ where 
\begin{equation}\label{def:d_j}
  d_j:=\frac{\log m_j}{-\log \frac{m_{j+1}\eps_{j+1}}{m_j\eps_j}}.  
\end{equation}
Our goal is that for each $\eps>0$, we make a construction of a 
  Cantor set of nonergodic directions so that each $d_j$ will satisfy 
  \begin{equation*}\label{eq:lowerd_j}
d_j > \frac{1}{2}-\eps.
\end{equation*}

\subsection{The parameters $r$, $\delta_j$, and $\rho_j$}
Given $r>1$ and a sequence of positive $\delta_j\to 0$ (which will 
  measure the area interchange defined by consecutive slits), we 
  shall construct a Cantor set $F$ depending on parameters $m_j$ 
  and $\eps_j$ that are expressible in terms of $r$ and $\delta_j$.  
It is based on the assumption, verified later, that we can construct 
  a tree of slits.  We start with an initial slit $w_0$, the unique 
  slit of level $0$.  Inductively, given a slit $w_j$ of level $j$ 
  we consider slits of the form $w_j+2v_j$ where $v_j\in\Z^2$ is a 
  primitive vector, i.e. $\gcd(v_j)=1$, and satisfies 
\begin{equation*}\label{def:children}
  |w_j\times v_j|<\delta_j, \quad |w_j|^r \le |v_j| \le 2|w_j|^r.  
\end{equation*}
We refer to $w_{j+1}=w_j+2v_j$ of the above form as a \emph{child} of $w_j$.  
It satisfies 
\begin{equation}\label{ieq:|w_{j+1}|}
  |w_j|^r\leq |w_{j+1}|\leq 5|w_j|^r.  
\end{equation}
The main difficulty in the construction is avoiding slits that have 
  no children at all.  To ensure that we can avoid such slits, we 
  shall only use children with ``nice Diophantine properties'' 
  when we assemble the slits for the next level.  
However, we shall ensure that at each stage, the number of 
  children (of a parent slit $w$) used will be at least 
\begin{equation}\label{enough}
  \rho_j|w|^{r-1}\delta_j  
\end{equation}
  where $\rho_j$ is to be determined later.  

For $w$ a slit, let $I(w)$ denote the interval of length 
  $$\diam I(w) = \frac{4}{|w|^{r+1}}$$ centered about the 
  inverse slope of the direction of $w$.  
The following lemma allows us to find estimates for the 
  sizes of intervals and the gaps between them.  
\begin{lemma}\label{lem:gaps}
Assume $|w_0|^{r(r-1)}\ge64$ and $\delta_j<\frac{1}{16}$.  
Let $w_{j+1}$ be a child of a slit $w_j$ of level $j$.  Then 
\begin{itemize}
\item $I(w_{j+1})\subset I(w_j)$, and 
\item if $w'_{j+1}$ is another child of $w_j$, then 
  $$\dist(I(w_{j+1}),I(w'_{j+1})) \ge \frac{1}{16|w_j|^{2r}}.$$  
\end{itemize}
\end{lemma}
\begin{proof}
Since the distance between the directions of $w_j$ and $w_{j+1}$ is 
  $$\angle w_jw_{j+1} = \frac{|w_j\times w_{j+1}|}{|w_j||w_{j+1}|} 
              \le \frac{|w_j\times v_j|}{|w_j||v_j|} < \frac{1}{|w_j|^{r+1}}$$ 
  the first conclusion follows from 
  $$\frac{1}{|w_j|^{r+1}} + \frac{2}{|w_j|^{r(r+1)}} \le \frac{2}{|w_j|^{r+1}}$$ 
  which holds easily by the assumption on $|w_0|$.  

The distance between the directions of $w_{j+1}$ and $v_j$ is 
  $$\angle w_{j+1}v_j = \frac{|w_{j+1}\times v_j|}{|w_{j+1}||v_j|} 
       \le \frac{|w_j\times v_j|}{|w_j|^{2r}} < \frac{\delta_j}{|w_j|^{2r}}.$$  
If $w'_{j+1}=w_j+2v_j'$ is another child of $w_j$ then 
  $$\angle v_jv_j'=\frac{|v_j\times v_j'|}{|v_j||v_j'|} \ge \frac{1}{4|w_j|^{2r}}$$ 
  so that by the triangle inequality, 
  $$\angle w_{j+1}w'_{j+1} \ge \frac{1}{4|w_j|^{2r}} 
       - \frac{\delta_j+\delta_{j+1}}{|w_j|^{2r}} \ge \frac{1}{8|w_j|^{2r}}$$ 
  since $\sup\delta_j<\frac{1}{16}$.  
Therefore, $$\dist(I(w_{j+1}),I(w'_{j+1})) \ge \frac{1}{8|w_j|^{2r}} 
       - \frac{4}{|w_j|^{r(r+1)}} \ge \frac{1}{16|w_j|^{2r}}$$ 
  since $|w_0|^{r(r-1)}\ge64$.  
\end{proof}

Let $$F_j=\bigcup_w I(w)$$ where the union is taken over all 
  slits of level $j$.  
From (\ref{ieq:|w_{j+1}|}) we have 
\begin{equation}\label{ieq:|w_j|}
  |w_0|^{r^j} \le |w_j| \le 5^{\frac{r^j-1}{r-1}}|w_0|^{r^j},
\end{equation}
  so that the number of children given by (\ref{enough}) is at least 
\begin{equation}\label{def:m_j}
  m_j := \rho_j\delta_j|w_0|^{r^j(r-1)} 
\end{equation}
  while the smallest gap between the associated intervals is at least 
\begin{equation*}\label{def:eps_j}
   \eps_j := \frac{1}{16\cdot5^{2r\frac{r^j-1}{r-1}}|w_0|^{2r^{j+1}}}, 
\end{equation*}
  by Lemma~\ref{lem:gaps}.  

Now we express $d_j$, given by (\ref{def:d_j}), in terms of $r,\delta_j$ 
  and $\rho_j$.  We have 
  $$m_j\eps_j = \frac{\rho_j\delta_j}{16\cdot5^{2r\frac{r^j-1}{r-1}}|w_0|^{r^j(r+1)}}$$ 
  so that 
  $$\frac{m_{j+1}\eps_{j+1}}{m_j\eps_j} = 
        \frac{\rho_{j+1}\delta_{j+1}/\rho_j\delta_j}{5^{2r^{j+1}}|w_0|^{r^j(r^2-1)}}$$ 
  giving 
\begin{align}
\notag
  d_j &= \frac{r^j(r-1)\log|w_0| + \log(\rho_j\delta_j)}
               {r^j(r^2-1)\log |w_0| + 2r^{j+1}\log5 -\log(\rho_{j+1}\delta_{j+1}/\rho_j\delta_j)}\\
\label{formula:d_j}
      &= \frac{1 - \frac{-\log(\rho_j\delta_j)}{r^j(r-1)\log|w_0|}}
              {1+r + \frac{2r\log5}{(r-1)\log|w_0|} 
                   + \frac{\log(\rho_j\delta_j/\rho_{j+1}\delta_{j+1})}{r^j(r-1)\log|w_0|}}.  
\end{align}

Now making $d_j$ close to $\tfrac12$ will mean making $r$ close to $1$ 
  and making the terms 
\begin{equation}\label{eq:num}
  \frac{-\log(\rho_j\delta_j)}{r^j(r-1)\log|w_0|}
\end{equation}
  and 
\begin{equation}\label{eq:den}
  \frac{2r\log5}{(r-1)\log|w_0|} + 
  \frac{\log(\rho_j\delta_j/\rho_{j+1}\delta_{j+1})}{r^j(r-1)\log|w_0|}
\end{equation}
  small. 
Notice that if $\rho_j$ and $\delta_j$ are constant sequences, then this is 
  easily accomplished by choosing $|w_0|$ large enough.  
In \S\ref{s:Tree} we shall show that $\delta_j$ and $\rho_j$ can be chosen 
  so that (\ref{enough}) is satisfied at each step of the construction.  
The conditions $\delta_j<\frac{1}{16}$, as required by Lemma~\ref{lem:gaps}, 
  and $m_j\ge2$, as required by Falconer's estimate, will be verified in \S\ref{s:Tree} 
  along with the fact that $|w_0|$ can be chosen large enough to ensure that 
  $d_j$ is close to $\tfrac12$.

\section{Liouville construction}\label{s:Liouville}
The slits of the next level will be constructed from the previous level using 
  one of two constructions.  
The first construction we call the \emph{Liouville construction} as it uses 
  the Liouville convergents of $\lambda$ directly to identify new slits.  
The second construction, introduced in \cite{Ch1}, is different.  
We call it the \emph{Diophantine construction}.  
It does not use directly the convergents of $\lambda$, but rather 
  employs a technique to count lattice points in certain strips.  

In this section, we begin with the Liouville construction as it is 
  perhaps the main one of the paper.  
The Diophantine construction will be explained in \S\ref{s:Diophantine}.  

Recall that the Liouville convergent of a slit $w=(\lambda+m,n)$ 
  indexed by $k$ is the vector $u\in\Z^2$ determined by 
\begin{equation}\label{Liouville:convergent}
  (p_k+mq_k,nq_k) = du, \quad \gcd(u)=1 
\end{equation}
  where $$d=d(w,k)=\gcd(p_k+mq_k,nq_k).$$  
Note that the height of the Liouville convergent satisfies $$d|u|=|w|q_k.$$  

Choose $\Tilde{u}\in\Z\times\Z_{>0}$ so that 
  $$|u\times\Tilde{u}|=1 \quad\text{ and }\quad |\Tilde{u}|\le|u|.$$  
Observe that there are exactly $2$ possibilites for $\Tilde{u}$.  

Let $$\Lambda_1(w,k)=\{w+2v : v = \Tilde{u}+au, a\in\Z_{>0}\}$$ consist 
  of children $w+2v$ such that $v$ forms a basis for $\Z^2$ together 
  with $u$, i.e. $\Z^2=\Z u+\Z v$.  
  
The next lemma gives a bound on the cross-product of a parent with a child, 
  which recall, is a necessary estimate in the construction of nonergodic directions. 
\begin{lemma}\label{lem:max:area}
If $w+2v\in\Lambda_1(w,k)$ for some $|v|<q_{k+1}$ then 
\begin{equation}\label{ieq:max:area}
  |w\times v| < \frac{2|w|}{|u|} = \frac{2d(w,k)}{q_k} 
\end{equation}
  where $u$ is the Liouville convergent of $w$ indexed by $k$.  
\end{lemma}
\begin{proof}
From (\ref{ieq:LC}) we have $$\angle uw\le\frac{1}{|w|q_kq_{k+1}}.$$ 
Since $|v|<q_{k+1}$, $|u\times v|=1$ and $|u|\le|w|q_{k}$ we have 
  $$\angle uv = \frac{|u\times v|}{|u||v|} > \frac{1}{|u|q_{k+1}} 
       \ge \frac{1}{|w|q_kq_{k+1}}$$
  so that $\angle vw\le \angle uv+\angle uw<2\angle uv$.  
Therefore, 
  $$|w\times v| = |w||v|\angle vw < 2|w||v|\angle uv 
                = \frac{2|w|}{|u|}.$$  
\end{proof}

The next lemma expresses the key property of slits constructed 
  via the Liouville construction.  
Note that $d(w,k)$ measures  how far $\frac{p+mq}{nq}$ is from 
  being a reduced fraction; namely, it is the amount of cancellation 
  between the numerator and denominator.  
Since $\gcd(p,q)=1$ (and $n=|w|$), it is easy to see that $d(w,k)\le|w|$.  
It is quite surprising that whenever a new slit $w'$ is constructed 
  via the Liouville construction, we have $d(w',k)\le2$.  
\begin{lemma}\label{lem:gcd}
For any $w'\in\Lambda_1(w,k)$, we have $d(w',k)\le2$.  
Hence, if $|w'|<\frac{q_{k+1}}{2q_k}$, then the inverse slope of $w'$ 
  has a convergent whose height is either $q_k|w'|$ or $q_k|w'|/2$.  
\end{lemma}
\begin{proof}
Let $w'=(\lambda+m',n')$ where $$(m',n')-(m,n)=w'-w=2v.$$  
Now $d'=d(w',k)$ is determined by $d'u'=(p_k+m'q_k,n'q_k)$ 
  for some primitive $u'\in\Z^2$.  
In terms of the basis given by $u$ and $\tilde{u}$ we have 
  $$d'u'=(p_k+mq_k,nq_k)+2q_k(\tilde{u}+au)=(2q_k)\tilde{u}+(2aq_k+d)u.$$  
Note that $d=\gcd(p_k+mq_k,nq_k)$ is not divisible by any 
  divisor of $q_k$, since $\gcd(p_k,q_k)=1$.  Therefore, 
  $$d' = \gcd(2aq_k+d,2q_k) = \gcd(d,2q_k) = \gcd(d,2) \le2.$$  
The second statement follows from Lemma~\ref{lem:LC}.  
\end{proof}

Given $r>1$ we let 
  $$\Lambda(w,k)=\{w+2v\in\Lambda_1(w,k) : |w|^r \le |v| \le 2|w|^r \}.$$  
The next lemma gives a lower bound for the number of children constructed in the Liouville construction. 
\begin{lemma}\label{lem:liouville}
If $|w|^{r-1}\ge q_k$ then 
\begin{equation}\label{ieq:liouville}
  \#\Lambda(w,k) \ge \frac{|w|^{r-1}}{q_k}.  
\end{equation}
\end{lemma}
\begin{proof}
Since there are $2$ choices for $\Tilde{u}$ 
  the number of slits in $\Lambda(w,k)$ is at least 
  $$\#\Lambda(w,k) \ge 2\left[\frac{|w|^r}{|u|}\right]
       \ge \frac{|w|^r}{|u|} \ge \frac{|w|^{r-1}}{q_k}$$ 
  where $|u|\le|w|q_k\le|w|^r$ was used in the last two inequalities.  
\end{proof}

\section{Diophantine construction}\label{s:Diophantine}
Now we explain  our next general construction, which is 
  accomplished by Proposition~\ref{prop:normal}.  
Many of the ideas in this section already appeared in \cite{Ch1}.  

Again given a parent slit $w$ we will construct new slits of the form $w+2v$, 
  where $v$ is a loop satisfying certain conditions on its length and cross-product with $w$.  
Not all of these solutions $w+2v$ will be used at the next level for it may happen that some 
  of these will not themselves determine enough further slits.  
In other words, we will only use some of the slits $w+2v$ of the parent $w$ and the ones 
  used will be called the \emph{children} of $w$.  
It will be encumbent to show that there are enough children at each stage in order to obtain 
  lower bounds on the Hausdorff dimension of the Cantor set of \S\ref{s:Cantor}.  
 
\subsection{Good slits}\label{ss:Good}
Assume parameters $1<\alpha<\beta$ be given.  In later sections they will each have a 
  dependence on the slit so they are not to be thought of as absolute constants.  

\begin{definition}
We say a slit $w$ is $(\alpha,\beta)$-\emph{good} if its inverse slope has 
  a convergent of height $q$ satisfying $\alpha|w|\le q\le \beta|w|$.  
\end{definition}

Let $\Delta(w,\alpha,\beta)$ be the collection of slits of the form $w+2v$ 
  where $v\in\Z\times\Z_{>0}$ satisfies $\gcd(v)=1$ and 
\begin{equation}\label{def:Delta}
  \beta|w| \le |v| \le 2\beta|w| \quad\text{and}\quad 
  \frac{1}{\beta} < |w\times v| < \frac{1}{\alpha}.  
\end{equation}
Notice the right hand inequality gives an upper bound for the cross product of $w$  with $w+2v$.
The next lemma gives a lower bound for the number of such $w+2v$ constructed from good slits $w$. 
\begin{lemma}\label{lem:good}
There is a universal constant $0<c_0<1$ such that 
\begin{equation}\label{number:good}
  \#\Delta(w,\alpha,\beta)\ge\frac{c_0\beta}{\alpha}.  
\end{equation}
  for any $(\alpha,\beta)$-good slit $w$ and $\alpha<c_0\beta$.  
\end{lemma}
\begin{proof}
By [Ch1,Thm.3], the number of primitive vectors satisfying 
\begin{equation}\label{def:Delta_1}
  \beta|w| \le |v| \le 2\beta|w| \quad\text{and}\quad 
  |w\times v| < \frac{1}{\alpha}.  
\end{equation}
  is at least $c'_0\beta/\alpha$ where $c'_0>0$ is some universal 
  constant.\footnote{To apply [Ch1,Thm.3] one needs to assume 
  $\beta\gg\alpha$, but this hypothesis was shown to be redundant in 
  \cite{Ch2}.  Indeed, by [Ch2,Thm.4] we can take $c'_0=\frac{4}{27\pi}.$} 
The angle, by which we mean the distance between inverse slopes, 
  between any two solutions $v,\hat v$ to (\ref{def:Delta_1}) is at least 
   $$\left|\frac{p}{q}-\frac{\hat p}{\hat q}\right| \ge \frac{1}{q\hat q} 
         \ge \frac{1}{4\beta^2|w|^2}.$$  
Take an interval $J$ of length $\frac{2}{\beta^2|w|^2}$ centered at the 
  inverse slope of $w$ and divide it into $8$ equal subintervals.  
The inequality above says that there is at most one solution $v$ whose 
  inverse slope lies in each subinterval.  Thus, by discarding at most $8$ 
  of these solutions, namely those with inverse slopes in $J$, 
  we can ensure that the remaining solutions satisfy 
  $$\frac{|w\times v|}{|w||v|} > \frac{1}{\beta^2|w|^2}.$$  
These solutions satisfy (\ref{def:Delta}) since 
  $$|w\times v| > \frac{|v|}{\beta^2|w|} \ge \frac{1}{\beta}.$$  

Let $c_0=c'_0/9$.  We may clearly assume $c'_0<9$ so that $c_0<1$.  
Since $\alpha<c_0\beta$, there are at least $c_0'\beta/\alpha>9$ 
  solutions to (\ref{def:Delta_1}).  
Of these, at least one satisfies (\ref{def:Delta}).  
Therefore, the number of primitive vectors satisfying (\ref{def:Delta}) 
  is at least 
  $$\frac{c'_0\beta}{\alpha}-8 
     \ge\left(9c_0-8\frac{\alpha}{\beta}\right)\frac{\beta}{\alpha} 
     \ge\frac{c_0\beta}{\alpha}.$$  
\end{proof}

\begin{lemma}\label{lem:good:children}
Let $w$ be an $(\alpha,\beta)$-good slit.  Then every 
  $w'\in\Delta(w,\alpha,\beta)$ is $(\alpha-\frac12,\beta)$-good, 
  but not $(1,\alpha-\frac12)$-good.  
\end{lemma}
\begin{proof}
Let $w'=w+2v\in\Delta(w,\alpha,\beta)$.  
Note that $v$ is a convergent of (the inverse slope of) $w'$ since, 
  writing $w'=(\lambda+m',n')$ and $v=(p,q)$, we have 
  $$\left|\frac{\lambda+m'}{n'}-\frac{p}{q}\right| = 
  \frac{|w'\times v|}{|w'||v|}<\frac{|w\times v|}{2|v|^2} 
    \frac{1}{2\alpha q^2}<\frac{1}{2q^2}$$ 
  and we can use (\ref{ieq:CF2}).  

Let $q'$ be the height of the next convergent of $w'$.  
Then by (\ref{ieq:CF1})
  $$\frac{1}{q(q'+q)} < \left|\frac{\lambda+m'}{n'}-\frac{p}{q}\right| 
     < \frac{1}{qq'}$$ so that 
\begin{equation}\label{ieq:cfc}
  \frac{1}{q'+q} < \frac{|w\times v|}{|w'|} < \frac{1}{q'}.
\end{equation}
From the left hand side above, the fact that $|w'|>2|v|=2q$ and 
  $|w\times v|<\frac{1}{\alpha}$, we have 
  $$q' > \frac{|w'|}{|w\times v|} - q > (\alpha-\frac{1}{2})|w'|.$$  
Now, from the right hand side of (\ref{ieq:cfc}), we have 
  $$q' < \frac{|w'|}{|w\times v|} < \beta|w'|.$$  
This shows that $w'$ is $(\alpha-\frac12,\beta)$-good.  

Since $q$ and $q'$ are the heights of consecutive convergents 
  of $w'$ (and since $|v|<|w'|$) 
  it follows that $w'$ is not $(1,\alpha-\frac12)$-good.  
\end{proof}

\subsection{Normal slits}\label{ss:Normal}
In this subsection, we assume $N>0$ is fixed and set 
\begin{equation}\label{def:ell_N}
  \ell_N=\{k:q_{k+1}>q_k^N\}.  
\end{equation}
The choice of the parameter $N$ will depend on considerations 
  in \S\ref{s:Init} and will be specified there, by (\ref{def:N}).  

Given $N>0$, we set 
\begin{equation}\label{def:N'}
  N' = \frac{(N+1)r}{r-1}.  
\end{equation}
It will also be convienent to set  $$\rho=r+1/2.$$  

\begin{definition}
A slit $w$ is $\alpha$-normal if 
  it is $(\alpha\rho^t,|w|^{(r-1)t})$-good for all $t\in[1,T]$ 
  where $T>1$ is determined by $\alpha \rho^T=|w|^{r-1}$.  
Equivalently, $w$ is $\alpha$-normal if and only if 
  for all $t\in[1,T]$ we have 
\begin{equation}\label{def:normal}
  \Psi(w)\cap[\alpha \rho^t|w|,|w|^{1+(r-1)t}]\neq\emptyset.  
\end{equation}
  where $\Psi(w)$ denotes the collection of heights of the 
  convergents of the inverse slope of $w$.  
\end{definition}

The following gives a sufficient condition for a slit to be normal.  
\begin{lemma}\label{lem:N'-good=>normal}
Let $w$ be  a slit such that 
  $$q_{k+1}^{1/N}\le|w|<q_{k'}^{1/r}$$ 
  where $k,k'$ are consecutive elements of $\ell_N$ for some
 $N>0$.  
If $w$ is $(\alpha \rho^{N'},|w|^{r-1})$-good then it is $\alpha$-normal.  
\end{lemma}
\begin{proof}
Suppose on the contrary that $w$ is $(\alpha\rho^{N'},|w|^{r-1})$-good 
  but not $\alpha$-normal.\footnote{We remark that $N'>1$ implies 
  $T>1$ in the definition of normality.}
Let $\frac{p}{q}$ be the convergent of the inverse slope of $w$ with 
  maximal height $q\le|w|^r$.  
Since $w$ is $(\alpha \rho^{N'},|w|^{r-1})$-good, we have 
  $$\alpha \rho^{N'}|w| \le q \le |w|^r.$$
Let $q'$ the height of the next convergent.  
If $q'\le|w|^{1+(r-1)N'}$ then (\ref{def:normal}) is satisfied by $q$ 
  for all $t\in[1,N']$, and by $q'$ for all $t\in[N',T]$.  
Since $w$ is not $\alpha$-normal we must have $$q'>|w|^{1+(r-1)N'}.$$  
Note that $$\frac{q'}{|w|} > |w|^{(r-1)N'} \ge q^{(1-r^{-1})N'} = q^{N+1}.$$  
Writing $w=(\lambda+m,n)$ we have
  $$\left|\frac{\lambda+m}{n}-\frac{p}{q}\right|<\frac{1}{qq'}$$
  so that 
  $$\left|\lambda+m-\frac{np}{q}\right|<\frac{|w|}{qq'}<\frac{1}{q^{N+2}}<\frac{1}{2q^2}$$
  from which it follows, by (\ref{ieq:CF2}), 
  that $m-\frac{np}{q}$ is a convergent of $\lambda$, say 
  $$\frac{p_h}{q_h}=m-\frac{np}{q}.$$  
Since, by (\ref{ieq:CF1}), $$\frac{1}{2q_h q_{h+1}} 
    < \left|\lambda-\frac{p_h}{q_h}\right| < \frac{1}{q^{N+2}}$$ 
  we have $$q_{h+1} > \frac{q^{N+2}}{2q_h} > q^N \ge q_h^N,$$ 
  from which it follows that $h\in\ell_N$.  
Since $q_h\le q\le|w|^r<q_{k'}$, we must have $q_h\le q_k$.  
Hence, $q_{h+1}\le q_{k+1}$ so that 
  $$\frac{1}{2q_kq_{k+1}}\le\frac{1}{2q_h q_{h+1}} < \frac{1}{q^{N+2}}.$$  
Since $\alpha>1$, we have $q>|w|\ge q_{k+1}^{1/N} > q_k$ so that 
  $$q_{k+1} > \frac{q^{N+2}}{2q_k} > q^N > |w|^N,$$ 
  which contradicts the hypothesis on $|w|$.  
\end{proof}

Given a slit $w$ let $$\beta=|w|^{r-1}.$$
Our goal, Proposition~\ref{prop:normal}, is to develop hypotheses on an $\alpha$-normal slit $w$ that 
  ensures that among the slits $w'=w+2v\in \Delta(w,\alpha,\beta)$ lots of them are $\alpha r$-normal.  
More specifically we wish to show that under suitable hypotheses, an $\alpha$-normal slit $w$ determines 
  lots of $\alpha r$-normal $w'=w+2v$ where $v\in V_0$ and 
\begin{equation}\label{def:child}
 |w|^r\le|v|\le2|w|^r \quad\text{ and }\quad |w\times v|<\frac{1}{\alpha}
\end{equation}  
If  $w'$ is $\alpha r$-normal and satisfies (\ref{def:child}) then it will be called a {\em child} of $w$.  
The main task will be to bound the number of $w'$ that satisfy (\ref{def:child}) but are not $\alpha r$-normal.  
We begin with a pair of lemmas that are essentially a consequence of normality. 
\begin{lemma}\label{lem:t_1}
Suppose $w$ is $\alpha$-normal. Let 
 $u$ be the convergent of the inverse slope with maximum height $|u|< |w|^r$
and $q$ the height of the next convergent. Define $t_1$ by $$|u|=\alpha\rho^{t_1}|w|$$ and $t_2$ by $$q=|w|^{1+(r-1)t_2}.$$
 Then $1\leq t_1\leq T$ and $1\leq t_2\leq t_1$.
\end{lemma}
\begin{proof}
  The first inequality is a consequence of the case $t=1$ in the definition of normality 
  The left hand part of the second inequality follows from the defintion of $q$, while the right hand  
  follows from $\alpha$-normality because there 
  would otherwise be a $t\in(t_1,t_2)$ for which (\ref{def:normal}) fails.  
\end{proof}

\begin{lemma}\label{lem:t_1'<N'}
Suppose $w'\in\Delta(w,\alpha,\beta)$ is not $\alpha r$-normal and again letting $u'$ be the convergent of the inverse slope with maximum height $|u'|< |w'|^r$  and $q'$ the next convergent  define $t_1'$ by 
$|u'|=\alpha r\rho^{t_1'}|w'|$ and $t_2'$ by $q'=|w'|^{1+(r-1)t_2'}$.
	 Suppose $q_{k+1}^{1/N}<|w'|<q_{k'}^{1/r}$.  Then  $t_1'\leq \min(N',t_2')$.  
\end{lemma}
\begin{proof}
We first note again that $t_2'\geq 1$ by definition. Now $t_2'\geq t_1'$ since if $t_1'>1$  then (\ref{def:normal}) is satisfied 
 by $|u'|$ for all $t\in[1,t_1']$, and by $q'$ for all $t\in[t_1',T]$, contrary to the assumption that $w'$ is not $\alpha r$ normal.   If $t_1'\ge N'$ then Lemma~\ref{lem:N'-good=>normal} applied to 
  the slit $w'$, with $(\alpha r)$ in place of $\alpha$, implies that 
  $w'$ is $\alpha r$-normal, contrary to assumption.  
\end{proof}

\begin{lemma}\label{lem:strips}
Suppose $w'\in \Delta(w,\alpha,\beta)$ satisfies the conditions of Lemma~\ref{lem:t_1'<N'}.  Let $\bar {t}_1' := \max(t_1',1)$. 
 Let $u'$ be the convergent of $w'$ as above  
Then $u'$ determines a (nonzero) integer $a\in\Z$ such that 
  $$|(w\times u')+2a|<\frac{1}{|w|^{r(r-1)\Bar{t}_1'}}.$$  
Moreover, $|a|<2\rho^{N'+1}$.  
\end{lemma}
\begin{proof}
Write $w'=w+2v$ and recall that since $|w'\times v|=|w\times v|<1$ 
  (as in the proof of Lemma~\ref{lem:good}) $v$ is a convergent of $w'$.  
Let $v'$ be the next convergent of $w'$ after $v$.  
Since $|u'|>|v|$ we either have $u'=v'$ or $u'$ comes after $v'$ in the 
  continued fraction expansion of $w'$.  In any case, we have $u'=av'+bv$ 
  for some \emph{nonnegative} integers $a\ge b\ge0$ with $\gcd(a,b)=1$.  
Since $v\times v'=\pm1$ we have 
  $$|w'\times u'| = |(w\times u') + 2(v\times u')| = |(w\times u') \pm 2a|.$$  
On the other hand, $$|w'\times u'| 
    < \frac{|w'|}{q'}=\frac{1}{|w'|^{(r-1)t_2'}} < \frac{1}{|w|^{r(r-1)\Bar{t}_1'}}.$$  
This proves the first part.  

By the first inequality in (\ref{ieq:cfc}), $|v'|>\frac{|w'|}{2|w\times v|}$ so that 
  $$a<\frac{|u'|}{|v'|} < 2\alpha r\rho^{t_1'}|w\times v|<2\rho^{N'+1},$$ by 
  Lemma~\ref{lem:t_1'<N'} and since $r<\rho$.  
This proves the second part.  
\end{proof}

Suppose $w''\in\Delta(w,\alpha,\beta)$ also satisfies (\ref{def:child}) and  is 
  also not $(\alpha r)$-normal and satisfies $|w''|<q_{k'}^{1/r}$.  
Let $u''$ be the convergent of $w''$ with maximal height $|u''|\le|w''|^r$.  
Suppose further that it determines the same integer $a$ determined by 
  $u'$ as in Lemma~\ref{lem:strips}.  
Then we say $u'$ and $u''$ belong to the same \emph{strip}.  
The number of strips is bounded by the number of possible values for $a$.  
Thus, by Lemma~\ref{lem:strips}, the number of strips is bounded by 
\begin{equation}\label{bound:strips}
  4\rho^{N'+1}.  
\end{equation}

Now suppose $u',u''$ belong to the same strip.  We say $u'$ and $u''$ 
  lie in the same \emph{cluster} if they differ by a multiple of $u$.  
\begin{lemma}\label{lem:same:cluster}
If $|u''-u'|<|w|^r$ then they belong to the same cluster.  
\end{lemma}
\begin{proof}
Since $u',u''$ determine the same $a$, Lemma~\ref{lem:strips} and the 
  fact that $\bar{t}_1'\geq 1$ implies 
\begin{equation}\label{ieq:width:strip}
  |w\times(u''-u')|<\frac{2}{|w|^{r(r-1)}}
\end{equation}
  so that writing $u''-u'=d\Bar{u}$ where $d=\gcd(u''-u')$ we have 
  $$\angle w\Bar{u} = \frac{|w\times(u''-u')|}{|u''-u'||w|} 
        < \frac{2}{|u''-u'||w|^{r+(r-1)^2}} \le \frac{1}{2|u''-u'|^2},$$ 
  which implies $\Bar{u}$ is a convergent of $w$.  
Since $|\Bar{u}|\le|u''-u'|\le|w|^r$, we have $|\Bar{u}|\le|u|$, 
  by definition of $u$.  
Now suppose $|\Bar{u}|<|u|$.  We will arrive at a contradiction.  
Since $u$ is a convergent of $w$ coming after $\Bar{u}$, 
  $$|w\times\Bar{u}|>\frac{|w|}{2|u|}$$
  which together with (\ref{ieq:width:strip}) implies 
  $$\frac{d|w|}{2|u|} < \frac{2}{|w|^{r(r-1)}}$$ 
  so that $$|u|>\frac{d|w|^{r+(r-1)^2}}{4} \ge |w|^r,$$ 
  contradicting the definition of $u$.  
We conclude that $\Bar{u}=u$, so that $u',u''$ differ by a multiple of $u$.  
That is, they belong to the same cluster.  
\end{proof}

Pick a representative from each cluster.  To bound the number of 
  clusters we bound the number of representatives.  
Since $|u'|=\alpha r\rho^{t_1'}|w'| < 5\alpha\rho^{N'+1}|w|^r$ and the 
  difference in height of any two representatives is greater than $|w|^r$, 
  the number of clusters is bounded by (since $\alpha>1$) 
\begin{equation}\label{bound:clusters}
  5\alpha\rho^{N'+1}+1 \le 6\alpha\rho^{N'+1}.  
\end{equation}

To bound for the number of $u'$ in each cluster we need 
  an additional assumption.  
\begin{lemma}\label{lem:cluster:size}
Suppose $t_1'\ge t_1-1$ (independent of $u'$ within the cluster).  
Then the number of elements in the cluster is bounded by 
  $$5|w|^{(r-1)-(r-1)^2}.$$ 
\end{lemma}
\begin{proof}
Lemma~\ref{lem:strips} implies for any $u',u''$ in the cluster 
  $$|w\times(u''-u')|\le\frac{2}{|w|^{r(r-1)\Bar{t}_1'}}$$ 
  where $\Bar{t}_1'$ is the smallest possible within the cluster.  
On the other hand, $$|w\times u| > \frac{|w|}{2q} 
    = \frac{1}{2|w|^{(r-1)t_2}}\ge\frac{1}{2|w|^{(r-1)t_1}}$$ 
  so that 
  $$\frac{|w\times(u''-u')|}{|w\times u|}<4|w|^{(r-1)(t_1-r\Bar{t}_1')}.$$
By definition  $u''-u'$ is a multiple of $u$.   To get the desired bound, 
  using the assumptions $1<r<2$ and $|w|^{r-1}\geq 1$, 
  it remains to show that $$t_1-r\Bar{t}'_1 \le 2-r.$$  
To see this note that if $t_1'>1$ then since $t_1'\ge t_1-1$ 
  $$t_1-rt_1' = (t_1-t_1')+(1-r)t_1'<2-r,$$
  whereas if $t_1'\le1$ then 
  $$t_1-r \le t_1'+1-r \le 2-r.$$
\end{proof}

We shall now apply our Lemmas to show that, under suitable hypotheses on an 
  $\alpha$-normal slit $w$ there are lots of children, i.e. $\alpha r$-normal slits 
  $w'$ satisfying (\ref{def:child}).  
\begin{proposition}\label{prop:normal}
Suppose $w$ is an $\alpha$-normal slit satisfying 
  $$q_{k+1}^{1/N}\le|w|<5|w|^r<q_{k'}^{1/r}$$ 
  where $k,k'$ are consecutive elements of $\ell_N$.  
Suppose further that 
\begin{equation}\label{ieq:normal}
  240\alpha^2\rho^{3N'+3} \le c_0|w|^{(r-1)^2}.  
\end{equation} 
Then the number of $w'$ satisfying (\ref{def:child}) that are 
  $\alpha r$-normal is at least 
  $$\frac{c_0|w|^{r-1}}{2\alpha\rho^{N'+1}}$$
\end{proposition}
\begin{proof}
Let $t_1$ be the parameter associated to the convergent $u$ of $w$ 
  as in (\ref{lem:t_1}).  There are two cases.  
If $t_1\ge N'+1$ then $w$ is $(\alpha\rho^{N'+1},|w|^{r-1})$-good, 
  so that Lemma~\ref{lem:good} implies $w$ has at least 
\begin{equation}\label{double}
  \frac{c_0|w|^{r-1}}{\alpha\rho^{N'+1}}
\end{equation} 
  $w'=w+2v$ satisfying (\ref{def:child}).  Moreover, by Lemma~\ref{lem:good:children} 
  each $w'$ constructed is $(\alpha\rho^{N'+1}-\frac12,|w'|^{r-1})$-good.  
Since $$\alpha\rho^{N'+1}-\frac12>\alpha r\rho^{N'},$$ by the choice 
  of $\rho$, every such $w'$ is $(\alpha r\rho^{N'},|w'|^{r-1})$-good. 
  
Moreover, since each $w'$  has length at most $5|w|^r$, 
  Lemma~\ref{lem:N'-good=>normal} implies each $w'$ 
  constructed is $\alpha r$-normal.  
Note that the number in (\ref{double}) is twice as many as we need.  

Now consider the case $t_1<N'+1$.  
In this case $w$ is $(\alpha\rho^{t_1},|w|^{r-1})$-good, so that 
  Lemma~\ref{lem:good} implies $w$ has at least 
  $$\frac{c_0|w|^{r-1}}{\alpha\rho^{t_1}}>\frac{c_0|w|^{r-1}}{\alpha\rho^{N'+1}}$$ 
  $w'$ satisfying  (\ref{def:child}).   Moreover, Lemma~\ref{lem:good:children} implies 
  each child $w'$ constructed is $(\alpha\rho^{t_1}-\frac12,|w'|^{r-1})$-good, 
  and since $$\alpha\rho^{t_1}-\frac12>\alpha r\rho^{t_1-1},$$ again, by the 
  choice of $\rho$, this means $w'$ is $(\alpha r\rho^{t_1-1},|w'|^{r-1})$-good.  

Moreover, the parameter $t_1'$ associated to the convergent $u'$ of 
  each such $w'$ satisfies $t_1'\ge t_1-1$.  
Applying Lemmas~\ref{lem:strips}, \ref{lem:same:cluster} and \ref{lem:cluster:size} 
  we conclude the number of $w'$  constructed that are not $\alpha r$-normal 
  is at most the product of the bounds given in (\ref{bound:strips}), (\ref{bound:clusters}), 
  and Lemma~ (\ref{lem:cluster:size}), i.e. $$120\alpha\rho^{2N'+2}|w|^{(r-1)-(r-1)^2},$$ 
  which is at most half the amount in (\ref{double}) since  (\ref{ieq:normal}) holds.  
\end{proof}

\section{Choice of initial parameters}\label{s:Init}
In this section we specify some parameters that need to be 
  fixed before the construction of the tree of slits can begin.  
In particular, we shall specify the initial slit.  
We shall also specify the type of construction that will be 
  used at each level to find the slits of the next level.  

\subsection{Choice of initial slit}\label{ss:w_0}
Given $\eps>0$ we first choose $1<r<2$ so that 
  $$\frac{1}{1+r}>\frac{1}{2}-\eps$$ 
  then choose $\delta>0$ so that 
\begin{equation}\label{eq:local}
  \frac{1-\delta}{1+r+2\delta}>\frac{1}{2}-\eps.  
\end{equation}
It will be convenient to set $$M:=\frac{1}{r-1}>1$$ and let 
\begin{equation}\label{def:M'}
  M'=\max(3M^2,Mr/\delta).  
\end{equation}
We set 
\begin{equation}\label{def:N}
  N=M'r^5
\end{equation}
 and let $N'$ be given by (\ref{def:N'}).  

We assume that $\ell_N$, which was defined in (\ref{def:ell_N}), has 
  infinitely many elements, for if $\ell_N$ were finite, then $\lambda$ 
  is Diophantine and this case has already been dealt with in \cite{Ch1}.  
Our argument would simplify considerably if we assume $\ell_N$ is 
  finite and it would essentially reduce to the one given in \cite{Ch1}.  

Now choose $k_0\in\ell_N$ large enough so that 
\begin{equation}\label{def:k_0}
  q_{k_0} > \max\left(5^M,60c_0^{-1}\rho^{N'+3},2\rho^{N'}(\log_r(M')+4),2^7\rho{N'}\right).  
\end{equation}

\begin{lemma}\label{lem:init:slit}
There is a slit $w_0\in V_2^+$ such that $d(w_0,k_0)\le2$ and 
\begin{equation}\label{ieq:init:slit}
   q_{k_0}^{M'} \le |w_0| < q_{k_0}^{M'r}.  
\end{equation}
\end{lemma}
\begin{proof}
Let $w\in V_2^+$ be any slit such that $|w|<q_{k_0}/2$.  
Choose $w_0\in\Lambda_1(w,k_0)$ with minimal height 
  satisfying the first inequality in (\ref{ieq:init:slit}).  
Lemma~\ref{lem:gcd} implies $d(w_0,k_0)\le2$.  
Let $u$ be the Liouville convergent of $w$ indexed by $k_0$.  
Its height $|u|\le q_{k_0}|w|\le q_{k_0}^2/2$.  Since consecutive 
  elements in $\Lambda_1(w,k_0)$ differ by $2u$, we have 
  $$|w_0|<q_{k_0}^{M'}+2|u|\le q_{k_0}^{M'}+q_{k_0}^2<q_{k_0}^{M'r}$$ 
  since $M'>2M$.  
\end{proof}

Choose $w_0$ satisfying the conditions of Lemma~\ref{lem:init:slit} 
  and let it be fixed for the rest of this paper.  
It is the unique slit of level $0$.  

Note that the choice of $k_0$ in (\ref{def:k_0}) gives various lower bounds 
  on the length of $w_0$, by virtue of the first inequality in (\ref{ieq:init:slit}).  
For example, since $M'>M$, the first relation in (\ref{def:k_0}) implies 
\begin{equation}\label{ieq:w_0>5} 
  |w_0|^{(r-1)^2} \ge q_{k_0}^{M'/M^2 }> q_{k_0}^{1/M} > 5.  
\end{equation}

\subsection{Choice of indices}\label{ss:indices}
Next, we shall specify for each level $j\ge0$ the type of construction that 
  will be applied to the slits of level $j$ to construct slits of the next level.  
(The same type of construction will be applied to all slits within the same level.)  
We shall define indices $j^A_k$ for each $k\in\ell_N$ with $k\ge k_0$ and 
  for $A\in\{B,C,D\}$ such that whenever $k<k'$ are consecutive elements 
  of $\ell_N$ we have (see Lemma~\ref{lem:indices}(i) below) 
  $$j^B_k < j^C_k < j^D_k < j^B_{k'}.$$  
For $j^C_k\le j < j^D_k$ we use the construction described in \S\ref{s:Liouville}, 
  while for all other $j$ we use the techniques described in \S\ref{s:Diophantine}.  
The precise manner in which these types of constructions will be applied is 
  described in the next subsection.  

The primary role of these indices is to ensure that various conditions on the lengths 
  of all slits in some particular level are satisfied.  (See Lemma~\ref{lem:range}.)  
Specifically, the conditions in Lemmas~\ref{lem:gcd} and \ref{lem:liouville} are needed 
  for the levels $j^C_k\le j\le j^D_k$ and 
those in Proposition~\ref{prop:normal} are 
  needed for the levels $j^D_k\le j\le j^B_{k'}$.  
It will also be important that the number of levels between $j^B_k$ and $j^C_k$ be 
  bounded (Lemma~\ref{lem:indices}.ii) whereas the number between $j^C_k$ and $j^D_k$ 
  (or between $j^D_k$ and $j^B_{k'}$) will generally not be bounded.  

Let $H_0=\{|w_0|\}$ and for $j>0$ set 
  $$H_j=\left[|w_0|^{r^j},5^{\frac{r^j-1}{r-1}}|w_0|^{r^j}\right]$$ 
  so that the lengths of all slits of level $j$ lie in $H_j$, by  (\ref{ieq:|w_{j+1}|}).  
\begin{lemma}\label{lem:H_j}
 For all $j\ge0$ 
\begin{equation}\label{ieq:H_j}
  \sup H_j < \inf H_{j+1} = (\inf H_j)^r.
\end{equation}
\end{lemma}
\begin{proof}
The condition $\sup H_j<\inf H_{j+1}$ is equivalent to $$5^{\frac{r^j-1}{r-1}}<|w_0|^{r^j(r-1)},$$ 
  which is implied by $$5^{r^j}<|w_0|^{(r-1)^2r^j},$$ which in turn is implied by (\ref{ieq:w_0>5}).  
\end{proof}

The choice of the indices $j^A_k$ will depend on the position of $H_j$ relative to 
  that of the following intervals: 
  $$I^C_k=\left[q_k^{M'},q_{k+1}^{1/r}\right),\qquad\text{and}\qquad 
      I^D_k=\left[q_{k+1}^{1/r^5},q_{k'}^{1/r}\right).$$  
Here, again, $k'$ is the element in $\ell_N$ immediately after $k$.  
These intervals overlap nontrivially and the overlap cannot be too small in 
  the sense that there are at least three consecutive $H_j$'s contained in it.  
\begin{lemma}\label{lem:IKCD3}
 For any $k\in\ell_N$ with $k\ge k_0$ 
\begin{equation}\label{ieq:IKCD3}
  \#\{j: H_j \subset I^C_k\cap I^D_k\}\ge3.  
\end{equation}
\end{lemma}
\begin{proof}
Note that $f(x)=\log_r\log_{|w_0|}(x)$ sends $x=\inf H_j$ to 
  a nonnegative integer and 
  $$f(q^a) = \frac{\log a+\log\log q - \log\log|w_0|}{\log r}.$$  
For any $q$ the image of $[q^a,q^b)$ under $f$ contains exactly 
  $\lfloor\log_r(b/a)\rfloor$ integers, all of them nonnegative 
  if $f(q^a)>-1$; or equivalently, if $|w_0|<q^{ar}$.  
Under this condition, the fact in  Lemma~\ref{lem:H_j} that $\inf H_{j+1}=(\inf H_j)^r$ implies 
  $$\#\{j\ge0: H_j\subset\left[q^a,q^b\right)\} 
       \ge \left\lfloor\log_r(b/a)\right\rfloor -1.$$ 
Since  $N\ge M'r^5$ and $q_{k'}\ge q_{k+1}>q_k^N$, we have 
  $$I^C_k\cap I^D_k = \left[q_{k+1}^{1/r^5},q_{k+1}^{1/r}\right)$$ 
  and since $q_{k+1}^{1/r^4}>q_k^{N/r^4}\ge q_k^{M'r}>|w_0|$, 
  (\ref{ieq:IKCD3}) follows.  
\end{proof}

By virtue of the fact that the quantity in (\ref{ieq:IKCD3}) is at least one, 
  we can now give two equivalent definitions of the index $j^A_k$.  
\begin{definition}
For $k<k'$ consecutive elements of $\ell_N$ with $k\ge k_0$, let 
\begin{align*}
  j^C_k &= \min\{j:H_j\subset I^C_k\} = \min\{j:\inf H_j\ge q_k^{M'}\}\\
  j^D_k &= \max\{j:H_{j+1}\subset I^C_k\} = \max\{j:\sup H_j<q_{k+1}^{1/r}\}\\
  j^B_{k'} &= \max\{j:H_j\subset I^D_k\} = \max\{j:\sup H_j<q_{k'}^{1/r}\} 
\end{align*}  
\end{definition}
Note that $j^C_{k_0}=0$ and that $j^B_{k_0}$ is not defined.  

The main facts about these indices are expressed in the next two lemmas.  
\begin{lemma}\label{lem:indices}
For any $k\in \ell_N$, $k\ge k_0$ 
\begin{enumerate}
   \item[(i)] $j^B_k < j^C_k < j^D_k \leq j^B_{k'}$
  \item[(ii)] $j^C_k\le j^B_k + \log_r(M')+4.$  
\end{enumerate}
\end{lemma}
\begin{proof}
For (i) we note that 
  $$\inf H_{j_k^B} \le \sup H_{j_k^B} \le q_k^{1/r} < q_k^{M'}$$ 
  so the first inequality follows by the (second) definition of $j_k^C$.  
From the first definitions of $j^C_k$ and $j^D_k$, we see that the 
  second inequality is a consequence of Lemma~\ref{lem:IKCD3}.  
The third inequality follows by comparing the second definitions 
  of $j^D_k$ and $j^B_{k'}$ and noting that $q_{k'}\ge q_{k+1}$.  

For (ii) first note that 
  $$\inf H_{j^B_k} = \left(\inf H_{j^B_k+1}\right)^{1/r} 
       \ge\left(\sup H_{j^B_k+1}\right)^{1/r^2} \ge q_k^{1/r^3}$$ 
  by Lemma~\ref{lem:H_j} and the second definition of $j^B_k$.  
Thus, we have 
  $$\inf H_{j^B_k+n} = \left(\inf H_{j^B_k}\right)^{r^n} 
        \ge q_k^{r^{n-3}} \ge q_k^{M'}$$ 
  where $n=\lceil\log_r(M')+4\rceil$.  
The second definition of $j^C_k$ now implies 
  $j^C_k\le j^B_k+n\le\log_r(M')+4$.  
\end{proof}

\begin{lemma}\label{lem:range}
For any slit $w$ of level $j$ we have 
\begin{enumerate}
  \item[(i)] $\ds \quad j^C_k \le j \le j^D_k \impl |w|\in I^C_k \impl q_k^M\le|w|<\frac{q_{k+1}}{2q_k}$
  \item[(ii)] $\ds \quad j^D_k \le j \le j^B_{k'} \impl |w|\in I^D_k \impl q_{k+1}^{1/N'}\le|w|<q_{k'}^{1/r}.$
\end{enumerate}
\end{lemma}
\begin{proof}
By definition, $\inf H_{j^C_k}\ge q_k^{M'}$ and $\sup H_{j^D_k} <q_{k+1}^{1/r}$, 
  giving the first implication in (i).  
Since $N\ge2Mr$ we have $$q_{k+1}^{1-1/r}>q_k^{N-N/r}\ge q_k^2>2q_k$$  
  so that $q_{k+1}^{1/r}<\frac{q_{k+1}}{2q_k}$.  
This, together with $M'\ge M$, implies the second implication in (i).  

For (ii) note that (\ref{ieq:IKCD3}) implies 
  $H_{j^D_k}\subset I^C_k\cap I^D_k$, giving the first implication, 
  while the second implication follows from $N'>r^5$.  
\end{proof}

\section{Tree of slits}\label{s:Tree}
In this section we specify exactly how the slits of level $j+1$ are constructed 
  from the slits of level $j$.  As before, we refer to any slit constructed from a 
  previously constructed slit $w$ as a \emph{child} of $w$.  
The parameters $\delta_j$ and $\rho_j$ are also specified in this section.  
At each step, we shall verify that the choice of $\delta_j$ and $\rho_j$ is such 
  that all cross-products of slits of level $j$ with their children are $<\delta_j$ 
  while the number of children is at least $\rho_j|w|^{r-1}\delta_j$, as required 
  by (\ref{enough}) in \S\ref{s:Cantor}.  

Depending on the type of construction to be applied, there will be various kinds 
  of hypotheses on all slits within a given level that we need to verify.  
These hypotheses can be one of two kinds.  The first kind involve inequalities on 
  lengths of slits and these will always be satisfied using Lemma~\ref{lem:range}.  
We will not check these hypotheses explicitly.  
The second kind is more subtle and involve conditions related to the continued 
  fraction expansions of the inverse slopes of slit directions.  The fact that we need 
  such hypotheses on slits is evident from Lemma~\ref{lem:good}, which is one of 
  the main tools we have for determining whether a slit will have lots of children.  

One of the main tasks of this section will be to check the required hypotheses of 
  the second kind at each step.  
For the levels between consecutive indices of the form $j^A_k$, these hypotheses 
  will hold by virtue of the results in \S\ref{s:Liouville} and \S\ref{s:Diophantine}.  
Special attention is needed to check the relevant hypotheses of the second kind 
  for the levels $j^A_k,k\in{B,C,D}$ when the type of construction used to find the 
  slits of the next level changes.  

In what follows, it will be implicitly understood that $k<k'$ denote consecutive 
  elements of $\ell_N$, with $k\ge k_0$.  If $k>k_0$, then $\Tilde{k}$ will denote 
  the element of $\ell_N$ immediately before $k$.  

\subsection{Liouville region}\label{ss:Liouville}
For the levels $j$ satisfying $j^C_k\le j<j^D_k$, the slits of level $j+1$ will be 
  constructed by applying Lemma~\ref{lem:liouville} to all slits of level $j$.  
In other words, the slits of level $j+1$ consist of all slits $w'\in\Lambda(w,k)$ 
  where $w$ is a slit of level $j$ and $v$ is a loop such that $w'=w+2v$.  

Recall that an initial slit $w_0$ has been fixed using Lemma~\ref{lem:init:slit}.  
Lemma~\ref{lem:max:area} implies the cross-products of $w_0$ with its children 
  are all less than $4/q_{k_0}$, while Lemma~\ref{lem:liouville} implies the number 
  children is at least $|w|^{r-1}/q_{k_0}$.  Therefore, we set 
 $$\delta_0 = \frac{4}{q_{k_0}} \quad\text{ and }\quad \rho_0=\frac14.$$  

For the levels $j^C_k<j<j^D_k$, we set 
  $$\delta_j = \frac{4}{q_k} \quad\text{ and }\quad \rho_j=\frac14.$$  

\begin{lemma}\label{lem:LR}
For $j^c_k<j\le j^D_k$, every slit $w$ of level $j$ satisfies $d(w,k)\le2$.  
Moreover, if $j<j^D_k$ then 
  the cross-products of each slit of level $j$ with its children are less than $\delta_j$ 
  and the number of children is at least $\rho_j|w|^{r-1}\delta_j$.  
\end{lemma}
\begin{proof}
Since all slits of level $j$ were obtained via the Liouville construction, 
  the first part follows from the first assertion of Lemma~\ref{lem:gcd}.  
Suppose $w$ is a slit of level $j$ with $j^C_k<j<j^D_k$.  
Lemma~\ref{lem:max:area} now implies the cross-products of $w$ with 
  its children are less than $4/q_k$, and the number of children is at 
  least $|w|^{r-1}/q_k$, by Lemma~\ref{lem:liouville}.  
\end{proof}

It will be convenient to set $$\alpha_k=\frac{q_k}{2\rho^{N'}}.$$  
\begin{lemma}\label{lem:LR->DR}
Every slit of level $j^D_k$ is $\alpha_k$-normal.  
\end{lemma}
\begin{proof}
Let $w$ be a slit of level $j^D_k$.  Since $H_{j^D_k}\subset I^C_k$, we have 
  $$2\alpha_k\rho^{N'} = q_k \le |w|^{r-1}.$$  
By Lemma~\ref{lem:LR}, we have $d(w,k)\le2$ and since $w$ was obtained via 
  the Liouville construction,  Lemma~\ref{lem:gcd} implies the inverse slope of $w$ 
  has a convergent with height between $q_k|w|/2$ and $q_k|w|$, or, by the 
  above, between $\alpha_k\rho^{N'}|w|$ and $|w|^r$.  
This means $w$ is $(\alpha_k\rho^{N'},|w|^{r-1})$-good, and therefore, 
  $\alpha_k$-normal, by Lemma~\ref{lem:N'-good=>normal}.  
\end{proof}

\subsection{Diophantine region}\label{ss:Diophantine}
For the levels $j$ satisfying $j^D_k\le j<j^B_{k'}$, the slits of level $j+1$ will be 
  constructed by applying Proposition~\ref{prop:normal} with the parameter 
  $\alpha=\alpha_kr^{j-j^D_k}$ to all slits $w$ of level $j$.  
In other words, the slits of level $j+1$ consist of all $\alpha r$-normal children 
  of all slits of level $j$, where $\alpha r=\alpha_kr^{j-j^D_k+1}$.  

For the levels $j^D_k\le j<j^B_{k'}$, we set 
  $$\delta_j = \frac{2\rho^{N'}}{q_kr^{j-j^D_k}} \quad\text{ and }\quad 
        \rho_j = \frac{c_0}{2\rho^{N'+1}}.$$  

\begin{lemma}\label{lem:DR}
For $j^D_k\le j\le j^B_{k'}$, every slit $w$ of level $j$ is $\alpha_kr^{j-j^D_k}$-normal.  
Morevover, if $j<j^B_{k'}$ then 
  the cross-products of each slit of level $j$ with its children are less than $\delta_j$ 
  and the number of children is at least $\rho_j|w|^{r-1}\delta_j$.  
\end{lemma}
\begin{proof}
The case $j=j^D_k$ of the first assertion follows from Lemma~\ref{lem:LR->DR} 
  while the remaining cases follow from Proposition~\ref{prop:normal}.  

For children constructed via Proposition~\ref{prop:normal} applied to 
  an $\alpha$-normal slit, the cross-products are less than $1/\alpha$, 
  which is $\delta_j$ if $\alpha=\alpha_kr^{j-j^D_k}$.  
The number of children is at least $$\frac{c_0|w|^{r-1}}{2\alpha_k\rho^{N'+1}} 
   = \frac{c_0r^{j-j^D_k}}{2\rho^{N'+1}}|w|^{r-1}\delta_j\ge\rho_j|w|^{r-1}\delta_j$$ 
  provided we verify that the inequality (\ref{ieq:normal}) holds, i.e. if 
\begin{equation}\label{ieq:normal:2}
  60q_k^2r^{2(j-j^D_k)}\rho^{N'+3}\le c_0|w|^{(r-1)^2}.  
\end{equation}
To check this inequality, we first note that $|w|\ge q_{k+1}^{1/r}>q_k^{N/r}>q_k^{M'}$ 
  so that $$|w|^{(r-1)^2} > q_k^{M'/M^2} \ge q_k^3,$$ since 
  $M'\ge3M^2$, by the first relation in (\ref{def:M'}).  
Next, we note that it is enough to check (\ref{ieq:normal:2}) 
  in the case $j=j^D_k$ since the left hand side increases by 
  a factor $r^2$ as $j$ increments by one, while the right hand 
  side increases by a factor $|w|^{(r-1)^3}>q_{k_0}^{3(r-1)}>5>r^2$.  
Moreover, since $|w|^{(r-1)^2}>q_k^3$, (\ref{ieq:normal:2}) in 
  the case $j=j^D_k$ follows from $60\rho^{N'+3}<c_0q_{k_0}$, 
  which is guaranteed by the second term in (\ref{def:k_0}).  
\end{proof}

\begin{lemma}\label{lem:DR->BR}
Every slit $w$ of level $j^B_{k'}$ is $(\alpha_k,|w|^{r-1})$-good.  
\end{lemma}
\begin{proof}
Let $w$ be a slit of level $j^B_{k'}$.  Lemma~\ref{lem:DR} implies 
  that $w$ is $\alpha$-normal for some $\alpha>\alpha_k$.  
By the case $t=1$ in the definition of normality, this means $w$ is 
  $(\alpha,|w|^{r-1})$-good, i.e. its inverse slope has a convergent 
  whose height is between $\alpha|w|$ and $|w|^r$.  
Since $\alpha>\alpha_k$ the height of this convergent is between 
  $\alpha_k|w|$ and $|w|^r$.  
Hence, $w$ is $(\alpha_k,|w|^{r-1})$-good.  
\end{proof}

\subsection{Bounded region}\label{ss:Bounded}
For the levels $j$ satisfying $j^B_k\le j< j^C_k, k>k_0$, the slits of 
  level $j+1$ will be constructed by applying Lemma~\ref{lem:good} 
  to all slits $w$ of level $j$ with the parameters 
\begin{equation}\label{alpha:beta}
  \alpha=\alpha_{\Tilde{k}}-\frac{j-j^B_k}{2} \quad\text{ and }\quad \beta=|w|^{r-1}.  
\end{equation}
In other words, the slits of level $j+1$ consist of all slits of the form $w+2v$ 
  where $w$ is a slit of level $j$ and $v\in\Delta(w,\alpha,\beta)$ where 
  $\alpha$ and $\beta$ are the parameters given in (\ref{alpha:beta}).  

For the levels $j^B_k\le j< j^C_k, k>k_0$, we set 
  $$\delta_j = \frac{4\rho^{N'}}{q_{\Tilde{k}}} \quad\text{ and } \quad 
        \rho_j = \frac{c_0}{2}.$$  

\begin{lemma}\label{lem:BR}
For $j^B_k\le j\le j^C_k$, every slit $w$ of level $j$ is $(\alpha_{\Tilde{k}}/2,|w|^{r-1})$-good.  
Morevover, if $j<j^C_k$ then 
  the cross-products of each slit of level $j$ with its children are less than $\delta_j$ 
  and the number of children is at least $rho_j|w|^{r-1}\delta_j$.  
\end{lemma}
\begin{proof}
First we note that every slit $w$ of level $j$ is $(\alpha,\beta)$-good, where 
  $\alpha$ and $\beta$ are the parameters given in (\ref{alpha:beta}).  
Indeed, for $j=j^B_k$ this follows from Lemma~\ref{lem:DR->BR} while 
  for $j^B_k<j\le j^C_k$ it follows from Lemma~\ref{lem:good:children}.  
Lemma~\ref{lem:indices}.ii and the third relation in (\ref{def:k_0}) imply 
  $$j-j^B_k\le j^C_k-j^B_k\le\log_r(M')+4\le\frac{\alpha_{\Tilde{k}}}{2}$$ 
  from which we see that the first assertion holds.  

For children constructed via Lemma~\ref{lem:good} applied to 
  an $(\alpha,\beta)$-good slit, the cross-products are less than $1/\alpha$, 
  which is $<\delta_j$, since $\alpha>\alpha_{\Tilde{k}}/2$.  
And since $\alpha\le\alpha_{\Tilde{k}}$, the number of children is at least 
  $$\frac{c_0|w|^{r-1}}{\alpha_{\Tilde{k}}} = \rho_j|w|^{r-1}\delta_j$$ 
  giving the second assertion.  
\end{proof}

Finally, for the levels $j=j^C_k$ with $k>k_0$, we set 
  $$\delta_j = \frac{8\rho^{N'}}{q_{\Tilde{k}}} \quad\text{ and } \quad 
        \rho_j = \frac{q_{\Tilde{k}}}{8\rho^{N'}q_k}.$$  
\begin{lemma}\label{lem:BR->LR}
For any slit $w$ of level $j=j^C_k$ with $k>k_0$, 
  the cross-products of $w$ with its children are less than $\delta_j$ 
  and the number of children is at least $\rho_j|w|^{r-1}\delta_j$.  
\end{lemma}
\begin{proof}
Suppose $w$ is a slit of level $j^C_k$ with $k>k_0$.  The case $j=j^C_k$ of 
  Lemma~\ref{lem:BR} implies $w$ is $(\alpha_{\Tilde{k}}/2,|w|^{r-1})$-good.  
Let $u$ be the Liouville convergent of $w$ indexed by $k$.  
By Lemma~\ref{lem:LC} the height $q'$ of the next convergent is 
  $$q'>\frac{q_{k+1}}{2} > q_{k+1}^{1/r} > \left(\sup H_{j^C_k}\right)^r \ge |w|^r.$$  
Since $w$ is $(\alpha_{\Tilde{k}}/2,|w|^{r-1})$-good, we must have 
  $|u|\ge\alpha_{\Tilde{k}}|w|/2$ so that, by Lemma~\ref{lem:max:area} 
  the cross-products of $w$ with its children are $$<\frac{2d(w,k)}{q_k} 
  = \frac{2|w|}{|u|} \le \frac{4}{\alpha_{\Tilde{k}}} = \frac{8\rho^{N'}}{q_{\Tilde{k}}}.$$
By Lemma~\ref{lem:liouville}, the number of children is at least 
  $|w|^{r-1}/q_k = \rho_j|w|^{r-1}\delta_j$.  
\end{proof}

The construction of the tree of slits is now complete.  

\section{Hausdorff dimension $1/2$}\label{s:Lower}
We gather the definitions of $\delta_j$ and $\rho_j$ (for $j>0$) in the table below.  
\begin{center}
\begin{align*}
  &            & j^B_k&\le j<j^C_k &&j^C_k & j^C_k<&j<j^D_k & & j^D_k\le j < j^B_{k'} \\[12pt]
 &\delta_j &&\frac{4\rho^{N'}}{q_{\Tilde{k}}} &&\frac{8\rho^{N'}}{q_{\Tilde{k}}} &&\frac{4}{q_k} &&\frac{2\rho^{N'}}{q_kr^{j-j^D_k}} \\[12pt]
 &\rho_j   &&\frac{c_0}{2}         && \frac{q_{\Tilde{k}}}{8\rho^{N'}q_k}           &&\frac{1}{4}       & &\frac{c_0}{2\rho^{N'+1}} \\[12pt]
 &\rho_j\delta_j &&\frac{2c_0\rho^{N'}}{q_{\Tilde{k}}} &&\frac{1}{q_k} &&\frac{1}{q_k} &&\frac{c_0/\rho}{q_kr^{j-j^D_k}} \\
\end{align*}
\end{center}

First, we verify the hypotheses needed for Falconer's estimate.  
Recall the definition $m_j=\rho_j|w_0|^{r^j(r-1)}\delta_j$ in (\ref{def:m_j}).  
\begin{lemma}
$\delta_j<\tfrac{1}{16}$ and $m_j\ge2$ for $j\ge0$.  
\end{lemma}
\begin{proof}
From the fourth relation in (\ref{def:k_0}) we see that 
  $\delta_j\le\frac{8\rho^{N'}}{q_{k_0}} < \frac{1}{16}$.  
For $j^C_k\le j<j^D_k$ we have $\rho_j\delta_j=1/q_k$ and 
  since $|w_0|^{r^{j^C_k}}\ge q_k^{M'}$, by the definition of $j^C_k$, 
  we have $|w_0|^{r^j(r-1)}\ge q_k^{M'/M}\ge q_k^{3M}$, from which 
  it easily get $$m_j\ge\frac{2\rho}{c_0}$$ and, in particular, $m_j\ge2$.  
For $j^D_k\le j < j^C_{k'}$ the expression $|w_0|^{r^j(r-1)}$ increases 
  faster than $\rho_j\delta_j$ decreases, so it is enough to check 
  the case $j=j^D_k$, for which, by the above, we have 
  $m_j\ge\frac{c_0}{\rho}m_{j-1}\ge2$.  
\end{proof}

Next, we obtain the lower bound on the Hausdorff dimension of $F$.  
Recall the expression for the local Hausdorff dimensions $d_j$ given in (\ref{formula:d_j}).  
The next lemma shows it is close to $\tfrac12$ by the choices made in \S\ref{s:Init}.  

\begin{lemma}
$\liminf_{j\to\infty} d_j>\frac12-\eps.$  
\end{lemma}
\begin{proof}
By (\ref{eq:local}) it is enough to show that the term (\ref{eq:num}) and both 
  of the terms in (\ref{eq:den}) are bounded by $\delta$.  
By the choice of $M'$ in (\ref{def:M'}), it would be enough to show that each 
  term is bounded by $\frac{Mr}{M'}$ for all large enough $j$.  
It will be convenient to write $$A_j\lesssim B_j$$ as an abbreviation for 
  $\liminf A_j\le\liminf B_j$.  

We consider the expression (\ref{eq:den}) first.  Using (\ref{ieq:w_0>5}) and 
  the fact that $M'>2M^2$ we see that the first term in (\ref{eq:den}) satisfies 
  $$\frac{2r\log5}{(r-1)\log|w_0|}\le\frac{2r}{M}<\frac{Mr}{M'}.$$  
From the last row of the table, we see that for $j\neq j^C_{k-1}$ we have 
  $$\frac{\rho_j\delta_j}{\rho_{j+1}\delta_{j+1}} \in \left\{1,\frac{\rho}{c_0},r, 
      \frac{1}{2\rho^{N'+1} r^{j^B_{k'}-j^D_k-1}},  \right\}$$ 
  while for $j=j^C_{k-1}$ we have 
  $$\frac{\rho_j\delta_j}{\rho_{j+1}\delta_{j+1}} = \frac{2c_0\rho^{N'}q_k}{q_{\Tilde{k}}}.$$  
Then, in the second case, we have 
  $$\frac{\log(\rho_j\delta_j/\rho_{j+1}\delta_{j+1})}{r^j(r-1)\log|w_0|} \le 
    M\left(\frac{\log q_k+\log 2c_0\rho^{N'}}{r^{j^C_k-1}\log|w_0|}\right) \lesssim \frac{Mr}{M'}$$ 
  since $\inf H_{j^C_k}\in I^C_k$; 
  in the first case the left hand side above is $\lesssim0$.  

We now turn to the expression (\ref{eq:num}). 
For $j^C_k\le j<j^D_k$ we have  
  $$\frac{-\log(\rho_j\delta_j)}{r^j(r-1)\log|w_0|} 
     \le M\left(\frac{\log q_k}{r^{j^C_k}\log|w_0|}\right)
     \le \frac{M}{M'}.$$  
Next consider $j^D_k\le j<j^B_{k'}$.  Using $jr^{-j}\log r\le1$, 
  we have 
\begin{align*}
  \frac{-\log(\rho_j\delta_j)}{r^j(r-1)\log|w_0|} 
  &\le M\left(\frac{\log q_k+(j-j^D_i)\log r+\log(\rho/c_0)}{r^j\log|w_0|}\right) \\
  &\lesssim \frac{M(\log(q_k)+1)}{r^{j^D_k}\log|w_0|} 
   \lesssim \frac{Mr^5\log q_k}{\log q_{k+1}} < \frac{Mr^5}{N} = \frac{M}{M'}.  
\end{align*}
Finally, we turn to the possibility that $j_i^B\leq j< j_i^C$ ($i\ge1$).  
Since $j^C_k-j^B_k\le \log_r(M')+4$, we have 
\begin{align*}
  \frac{-\log(\rho_j\delta_j)}{r^j(r-1)\log|w_0|} 
  &\le M\left(\frac{\log q_{\Tilde{k}}-\log(2c_0\rho^{N'})}{r^{j^B_i}\log|w_0|}\right) 
  \lesssim \frac{MM'r^4\log q_{\Tilde{k}}}{r^{j^C_k}\log|w_0|} \\
  &\le \frac{Mr^4\log q_{\Tilde{k}}}{\log q_k} < \frac{Mr^4}{N} < \frac{M}{M'} 
\end{align*}
  and the lemma follows.  
\end{proof}

The proof of Theorem~\ref{thm:dichotomy} will be complete with 
  the proof of the following lemma. 
\begin{lemma}\label{lem:cross}
If $\lambda$ satisfies (\ref{PM:conv}) then $F\subset\NE(P_\lambda)$.  
\end{lemma}
\begin{proof}
It suffices to check that $\sum\delta_j<\infty$ for in that case, 
  every sequence $\dots, w_j, v_j, w_{j+1},\dots$ constructed 
  above satisfies (\ref{ieq:sumx}) and $F\subset\NE(P_\lambda)$, 
  by Theorem~\ref{thm:sumx}.  
We break the sum into three intervals: $j^B_k\le j\le j^C_k$, 
  $j^C_k<j<j_k^D$, and $j_k^D\leq j<j^B_{k'}$.  

Let $n_k=\log_{q_k}q_{k+1}$ so that $q_{k+1}=q_k^{n_k}$.  
It follows easily from the definitions that 
  $$j^D_k-j^C_k < \log_r n_k < \frac{\log\log q_{k+1}}{\log r}$$ 
  so that (\ref{PM:conv}) implies 
  $$\sum_{k\in\ell_N}\sum_{j^C_k<j<j^D_k}|w_j\times v_j| \le 
    \frac{4}{\log r}\sum_{k\in\ell_N}\frac{\log\log q_{k+1}}{q_k}<\infty.$$
Since $j_i^C-j_i^B\le\log_r(M')+4$ we have 
  $$\sum_{k\in\ell_N}\sum_{j^B_k\le j\le j^C_k}|w_j\times v_j| 
     \le \sum_{k\in\ell_N} \frac{8\rho^{N'}(\log_r(M')+5)}{q_{\Tilde{k}}}<\infty.$$  
Finally, 
  $$\sum_{k\in\ell_N}\sum_{j^D_k\le j< j^B_{k'}}|w_j\times v_j| 
      \le \sum_{k\in \ell_N}\frac{2R\rho^{N'}}{q_k}<\infty$$  
  where $R=\sum_{j\ge0}r^{-j}$.  
\end{proof}

\begin{proof}[Proof of Theorem~\ref{thm:divergent}]
The construction of the set $F$ as well as the lower bound $1/2$ estimate on its Hausdorff dimension remains valid for any irrational $\lambda$.  (Note that when $\sum_k \frac{\log \log q_{k+1}}{q_k}=\infty$,  $F$ cannot be a subset of $NE(P_\lambda)$ since the latter has Hausdorff dimension $0$).  On the other hand  the fact that $\lim_{j\to\infty}\delta_j=0$ implies $F\subset\DIV(P_\lambda)$, 
  by [Ch2,Prop.~3.6].  
Therefore, $\Hdim\DIV(P_\lambda)\ge\tfrac12$ for all irrational $\lambda$.  
The opposite inequality follows from a more general result in \cite{Ma2}.  
Lastly, when $\lambda\in\Q$, the set $\DIV(P_\lambda)$ is countable, 
  so that its Hausdorff dimension vanishes.  
\end{proof}

\end{document}